\RequirePackage{amsthm}
\documentclass[sn-mathphys-num]{sn-jnl}
\usepackage{graphicx}%
\usepackage{multirow}%
\usepackage{amsmath,amssymb,amsfonts}%
\usepackage{amsthm}%
\usepackage{mathrsfs}%
\usepackage[title]{appendix}%
\usepackage{textcomp}%
\usepackage{manyfoot}%
\usepackage{booktabs}%
\usepackage{listings}%
\usepackage{anyfontsize} 

\usepackage{subfiles}
\usepackage{custom}

\usepackage{pifont}
\newcommand{\cmark}{\ding{51}}%
\newcommand{\xmark}{\ding{55}}%

\begin{document}
\title{Riemannian Adaptive Regularized Newton Methods with Hölder Continuous Hessians}

\author[1]{\fnm{Chenyu} \sur{Zhang}}\email{zcysxy@mit.edu}

\author*[2]{\fnm{Rujun} \sur{Jiang}}\email{rjjiang@fudan.edu.cn}

\affil[1]{\orgdiv{Institute for Data, Systems, and Society \& Laboratory for Information and Decision Systems}, \orgname{Massachusetts Institute of Technology}, \orgaddress{\state{Massachusetts}, \country{USA}}}
\affil[2]{\orgdiv{School of Data Science}, \orgname{Fudan University}, \orgaddress{\state{Shanghai}, \country{China}}}

\abstract{
	This paper presents strong worst-case iteration and operation complexity guarantees for Riemannian adaptive regularized Newton methods, a unified framework encompassing both Riemannian adaptive regularization (RAR) methods and Riemannian trust region (RTR) methods. We comprehensively characterize the sources of approximation in second-order manifold optimization methods: the objective function's smoothness, retraction's smoothness, and subproblem solver's inexactness. Specifically, for a function with a $\mu$-\holder continuous Hessian, when equipped with a retraction featuring a $\nu$-\holder continuous differential and a $\theta$-inexact subproblem solver, both RTR and RAR with $\ar{\alpha}$ regularization (where $\alpha=\min\{\mu,\nu,\theta\}$) locate an $(\epsilon,\epsilon^{\alpha /(1+\alpha)})$-approximate second-order stationary point within at most $O(\epsilon^{-(2+\alpha)/(1+\alpha)})$ iterations and at most $\widetilde{O}(\epsilon^{- (4+3\alpha) /(2(1+\alpha))})$ Hessian-vector products with high probability. These complexity results are novel and sharp, and reduce to an iteration complexity of $O(\epsilon^{-3 /2})$ and an operation complexity of $\widetilde{O}(\epsilon^{-7 /4})$ when $\alpha=1$.
}


\maketitle

\section{Introduction} 
We consider an unconstrained nonconvex manifold optimization problem:
$$
	\min_{x\in\M} f(x),
$$
where $\M$ is a complete Riemannian manifold, and $f$ is bounded below, and exhibits $C^{2,\mu}$ smoothness, meaning it is twice continuously differentiable and possesses a $\mu$-order \holder continuous Hessian.
Considering the problem's nonconvexity, we propose employing Riemannian adaptive regularized Newton (RARN) methods
that incorporate Riemannian trust region (RTR) \citep{absilTrustRegionMethodsRiemannian2007,zhang2023RiemannianTrust} and adaptive regularization (RAR) methods  \citep{agarwal2021Adaptiveregularization,qi2011Numericaloptimization}.
Within each iteration of RARN, we solve a \textit{regularized} model problem \textit{inexactly} on a tangent space:
$$
	\eta_k \approx \argmin_{\eta\in\tm{x}} \bar{m}_{x_k}(\eta)\coloneqq m_{x_k}(\eta) + \phi(\eta;\sigma_k),
$$
where $m_{x_k}$ is the Newton model function, i.e., the truncated second-order Taylor expansion of the objective function at $x_k\in\M$, and $\phi$ is the regularization, with $\sigma_k$ serving as an adaptively changing regularization parameter.
Subsequently, if the solution yields a significant decrease in the objective function, we pull it back onto the manifold using a \textit{retraction} as the next iteration point:
$$
	x_{k+1} = R_{x_k}(\eta_k).
$$
In an Euclidean space, the canonical retraction is the identity function, simplifying the iteration to $x_{k+1} = x_k + \eta_k$.
To align with the smoothness characteristics of the objective function, we employ $\theta$-inexact subproblem solvers and $C^{1,\nu}$ retractions.
Detailed definitions and discussions of these elements are deferred to \cref{sec:pre,sec:rarn}.

Motivated by recent advances in worst-case iteration complexity guarantees for Newton-type methods \citep{nesterov2006cubic,cartis2011Adaptivecubic,royer2020NewtonCGAlgorithm,curtis2021TrustRegionNewtonCG}, this paper investigates the iteration complexity of RTR and RAR under the aforementioned relaxations.
Iteration complexity refers to a \textit{non-asymptotic} bound on the number of outer iterations an algorithm needs to attain a solution within a given tolerance.
Concurrently, operation complexity, which bounds the number of \textit{unit operations} (either a gradient
evaluation or a Hessian-vector multiplication) required to find an approximated solution, has emerged as another focal non-asymptotic convergence measure due to its direct correspondence with computational cost \citep{curtis2021TrustRegionNewtonCG,agarwal2021Adaptiveregularization,carmon2020FirstOrderMethods}.
We also provide strong operation complexity guarantees, measured in terms of the number of Hessian-vector products, for both RTR and RAR.

\subsection{Main Results}

Our main results and contributions unfold across several dimensions.

\parag{A unified algorithm and iteration complexity analysis framework}
We begin by formulating a unified Riemannian adaptive regularized Newton (RARN) method framework, encompassing both Riemannian trust region (RTR) and adaptive regularization (RAR) methods.
Within this framework, we develop a unified iteration complexity analysis approach for RARN.
This versatile module allows us to derive iteration complexity guarantees for various regularization methods effortlessly.

\parag{A comprehensive approximation characterization and relaxation}
We identify the three sources of approximation within second-order manifold optimization methods:
\begin{enumerate*}[label=\upshape\arabic*\upshape)]
	\item The quadratic model function approximates the actual objective function.
	\item The retraction approximates the smooth diffeomorphism between the tangent space and the manifold;
	\item The inexact model problem solution approximates the exact solution.
\end{enumerate*}
These approximation levels can be characterized by the objective function's smoothness, retraction's smoothness, and subproblem solver's exactness.
Our investigation particularly explores objective functions with a $\mu$-\holder continuous Hessian, retractions with a $\nu$-\holder continuous differential, and $\theta$-inexact subproblem solvers.
Significantly, our analysis unveils an equitable contribution from these three sources of approximation to the iteration complexity: the iteration complexity order can only be improved by enhancing the worst-performing source.

\parag{A general-order adaptive regularization method}
Our work extends the realm of Riemannian adaptive regularization methods by introducing a novel adaptive $\ar{\omega}$ regularization method ($\omega\in(0,1]$), broadening the horizons beyond cubics. We pave the way for adaptivity without the constraint of $\mu \ge \omega$, a restriction posed by prior work on general-order adaptive regularization \citep{birgin2017Worstcaseevaluation,cartis2020SharpWorstCasea,cartis2011OptimalNewtontype},
in both Euclidean and Riemannian settings.
	Nevertheless, our findings reveal that optimal iteration complexity is achieved if and only if $\omega = \alpha \coloneqq \min \{ \mu,\nu,\theta \}$.

	\parag{A sharp worst-case iteration and operation complexity}
	Finally, we present our complexity guarantees.
	When $\omega = \alpha$, both RTR and RAR locate an $(\epsilon,\epsilon^{\alpha /(1+\alpha)})$-approximate second-order stationary point {(see \cref{def:opt})} within at most
$$
	O\left(\epsilon^{-\tfrac{2+\alpha}{1+\alpha}}\right)
$$
iterations.
This represents the first second-order iteration complexity result of Newton-type methods for objective functions with a \holder continuous Hessian, even within the framework of Euclidean spaces.
This complexity matches the optimal bound of $O(\epsilon^{-(2+\mu) /(1+\mu)})$ for this class of problems \cite{cartis2011OptimalNewtontype}, as well as the optimal bound of $O(\epsilon^{-3 /2})$ for $C^{2,1}$ problems \cite{carmon2020Lowerbounds} when $\mu = \nu = \theta = 1$.
Moreover, when $\alpha = 1$, this result establishes the first optimal iteration complexity result of $O(\epsilon^{-3 /2})$ for Riemannian trust region methods targeting both first-order and second-order stationarity.
Augmented by a meticulous analysis of the subproblem solvers, we also provide a Hessian-vector product operation complexity of $\widetilde{O}(\epsilon^{- (4+3\alpha) /(2(1+\alpha))})$, which reduces to $\widetilde{O}(\epsilon^{-7 /4})$ when $\alpha=1$.
\cref{tb:comp} compares our results with other related work and highlights the gaps our study addresses.

\ifSubfilesClassLoaded{\bibliography{ic}}{}
\subsection{Related Work}

\begin{table}[ht]
	\caption{~Comparison of worst-case iteration complexity results for adaptive regularized Newton methods, aiming to achieve $\epsilon$-approximate first-order stationarity or $(\epsilon,\epsilon^{\mu / (1+\mu)})$-approximate second-order stationarity.
		Asymptotic notation has been omitted from the complexity results.
		We remark that \cite{grapiglia2017RegularizedNewton} achieves the complexity (up to logarithmic factors) without requiring knowledge of $\mu$.
	}
	\label{tb:comp}
	\renewcommand{\arraystretch}{1.2}
	\centering
	{\begin{tabular}{cccccc}
			\toprule
			Work                                                                                       & Space      & Newton Extension           & Stationarity & Smoothness  & Iteration Complexity           \\
			\midrule
			\citet{grapiglia2017RegularizedNewton}                                                     & Euclidean  & line search                & first order  & $C^{2,\mu}$ & $\epsilon^{-(2+\mu)/(1+\mu)}$  \\
			\citet{royer2020NewtonCGAlgorithm}                                                         & Euclidean  & line search                & second order & $C^{2,1}$   & $\epsilon^{-3/2}$              \\
			\citet{curtis2021TrustRegionNewtonCG}                                                      & Euclidean  & trust region               & second order & $C^{2,1}$   & $\epsilon^{-3/2}$              \\
			\citet{nesterov2006cubic} and \citet{cartis2011Adaptivecubic}                              & Euclidean  & cubic regularization       & second order & $C^{2,1}$   & $\epsilon^{-3/2}$              \\
			\citet{boumal2019global}                                                                   & Riemannian & trust region               & second order & $C^{2,1}$   & $\epsilon^{-5/2}$              \\
			\rowcolor{cyan!30} Our work (\cref{cor:tr})                                                & Riemannian & trust region               & second order & $C^{2,1}$   & $\epsilon^{-3/2}$              \\
			\makecell{\citet{zhang2018CubicRegularized} and \citet{agarwal2021Adaptiveregularization}} & Riemannian & cubic regularization       & second order & $C^{2,1}$   & $\epsilon^{-3/2}$              \\
			\rowcolor{green!30} Our work (\cref{cor:ar})                                               & Riemannian & $2\!+\!\mu$ regularization & second order & $C^{2,\mu}$ & $\epsilon^{-(2+\mu) /(1+\mu)}$ \\
			\rowcolor{green!30} Our work (\cref{cor:tr})                                               & Riemannian & trust region               & second order & $C^{2,\mu}$ & $\epsilon^{-(2+\mu) /(1+\mu)}$
			\\\bottomrule
		\end{tabular}}
\end{table}

Absil et al.~\citet{absilTrustRegionMethodsRiemannian2007} extended trust region methods to Riemannian manifolds with asymptotic convergence results for $C^{2,1}$ objective functions.
Qi~\citet{qi2011Numericaloptimization} extended adaptive regularization with cubics to Riemannian manifolds with asymptotic local convergence results. This paper focuses on non-asymptotic convergence results of RTR and RAR.
We compare the related worst-case iteration complexity findings of adaptive regularized Newton methods in \cref{tb:comp}.
For simplicity, \cref{tb:comp} solely lists the smoothness requirement on the objective function. Note that our approach also introduces a more lenient smoothness requirement on the retraction ($R\in C^{1,\nu}$), distinguishing it from prevailing research on Riemannian methods where retractions are typically $C^{2}$ \cite{boumal2019global,zhang2018CubicRegularized,agarwal2021Adaptiveregularization}.

\cref{tb:comp} does not aim to encompass the entire research on the iteration complexity of Newton-type methods.
Given our specific focus on establishing complexity guarantees under relaxed smoothness requirements within the Riemannian setting, we only list the foundational works that initially established these complexities.
For more results along this line, please refer to \cite{carmon2018accelerated,carmon2020FirstOrderMethods,yao2023InexactNewtonCGa,xu2020Newtontypemethods}.
In addition to the research summarized in \cref{tb:comp}, it is noteworthy to mention related studies with distinct emphases.
Cartis et al.~\citet{cartis2020SharpWorstCasea} also considered a $\mu$-\holder continuous Hessian for a conceptual regularization algorithm, presenting a complexity of $O(\epsilon_{H}^{-(2+\mu) /\mu})$ for the attainment of an $\epsilon_{H}$-approximate second-order stationary point, aligning with our complexity (see \cref{cor:ar}).
Dedicated to functions with a $\mu$-\holder continuous Hessian, \cite{cartis2011OptimalNewtontype} proved that the complexity of finding an $\epsilon$-approximate first-order stationary point is lower bounded by $O(\epsilon^{- (2+\mu) /(1+\mu)})$, a bound which our methods successfully attain.

Beyond iteration complexity, recent research has introduced second-order methods capable of achieving a remarkable \textit{operation complexity} of $\widetilde O({\epsilon^{-7/4}})$ with high probability, for finding an $(\epsilon,\epsilon^{1/2})$-approximate stationary point.
Building upon adaptive regularization with cubics, Agarwal et al.~\cite{agarwal2017finding} derived an algorithm that attains this operation complexity bound, employing a subproblem solver with $\widetilde O(\epsilon^{-1/4})$ operation complexity. Subsequently, various methods have emerged that replicate the same complexity results \cite{royer2018Complexityanalysis,royer2020NewtonCGAlgorithm,jiang2022cubic}.
Very recently, Curtis et al.~\cite{curtis2021TrustRegionNewtonCG} demonstrated that such complexity can be achieved through a variant of trust region Newton methods that solve the trust region subproblem inexactly, utilizing the truncated conjugate gradient method \cite{steihaugConjugateGradientMethod1983,toint1981efficientsparsity}.
Carmon et al.~\cite{carmon2018accelerated} and Li and Lin \cite{li2022restarted} proposed variants of accelerated gradient methods that also converge to an $(\epsilon,\epsilon^{1/2})$-approximate stationary point with an {operation complexity} of $\widetilde O({\epsilon^{-7/4}})$.

\ifSubfilesClassLoaded{\bibliography{ic}}{}

\subsection*{Notation}
For scalars $x$ and $y$, we write $[x]_+ \coloneqq \max \{ 0,x \}$, $x\vee y \coloneqq \max \{ x,y \}$, and $x\w y \coloneqq \min \{ x,y \}$.
We write $x\gtrsim y$ if there exists a positive constant $C$ such that $Cx \ge y$, and write $x \lesssim y$ if $y\gtrsim x$.
For sets $\mathcal{A}$ and $\mathcal{B}$, we denote their cardinality by $|\mathcal{A}|$ and $|\mathcal{B}|$, and their disjoint union by $\mathcal{A}\sqcup \mathcal{B}$.
We sometimes write $f_k$, $g_k$, and $H_k$ as shorthand for the function value, gradient, and Hessian, respectively, of $f$ at $x_k$.
We also denote $f_0$ as the objective function evaluated at the initial point and $f_{\min}$ as the objective function's lower bound on the manifold.
Throughout our analysis, for any subscript, $C_{*}$ represents absolute positive constants independent of $\epsilon_{g}$ or $\epsilon_{H}$.
In algorithm descriptions, $(\mathrm{TC.X}\tcand\mathrm{Y})$ indicates the requirement for simultaneously satisfying both termination criteria $(\mathrm{TC.X})$ and $(\mathrm{TC.Y})$, while $(\mathrm{TC.X}\tcor\mathrm{Y})$ implies that fulfilling either termination criterion $(\mathrm{TC.X})$ or $(\mathrm{TC.Y})$ is sufficient.
We use \case{x} to tag different cases in the analysis for clarity.

\ifSubfilesClassLoaded{\bibliography{ic}}{}
\section{Preliminaries} \label{sec:pre}

We consider an unconstrained nonconvex optimization problem on a Riemannian manifold $\mathcal{M}$,
which is a smooth manifold equipped with an inner product $\left<\cdot,\cdot\right>_{x}$ on each tangent space $\tm{x}$, $x\in\M$, varying smoothly over $\M$. The Riemannian inner product further defines a norm $\|\cdot\|_{x}$ on $\tm{x}$ and a distance $\dist(\cdot ,\cdot )$ on $\M$.
Given our focus on the Riemannian metric, we will omit its subscript $x$ throughout the remainder of this paper.
The Riemannian inner product also helps define the gradient and Hessian of a function $f$ from $\M$ to $\R$:
$$
	\left<\g f(x),\xi \right> = \d f(x)[\xi],\quad
	\hess f(x)[\xi] = \nabla_{\xi}\g f(x), \quad \forall \xi\in\tm{x},
$$
where $\d f(x)$ is the differential of $f$ at $x$, and $\nabla$ is the Riemannian (Levi-Civita) connection on $\M$.

Throughout this paper, we consistently denote a parametrized smooth curve from $[0,1]$ to $\M$ as $\gamma$.
A vector field $X$ is said to be parallel along $\gamma$ if $\nabla_{\gamma'(t)}X = 0$ for all $t\in[0,1]$.
For any $\xi\in\tm{\gamma(0)}$, there exists a unique parallel vector field $X_{\xi}$ along $\gamma$ such that $X_{\xi}(0) = \xi$, defining a parallel transport operator $P_{\gamma}^{0\to t}: \xi \mapsto X_{\xi}(t)$.
Parallel transport is an isometry that preserves the inner product, bridging different tangent spaces.
Our focus is specifically on complete Riemannian manifolds, where geodesics---smooth curves whose tangent vector is parallel along itself---are the shortest curves connecting any two points and can be extended to the entire real axis.
Consequently, the exponential map $\exp_{x}:\xi\mapsto \gamma(1)$, where $\gamma$ is the unique geodesic with $\gamma(0)=x$, $\gamma'(0)=\xi$, and $\dist(\gamma(0),\gamma(1))=\|\xi\|$, is defined on the entire tangent space. When the parallel transport is along a geodesic from $x$ to $y$, we denote $P_{xy}\coloneqq P_{\gamma}^{0\to1}$. Detailed preliminaries for Riemannian optimization can be found in monographs such as \citet{absilOptimizationAlgorithmsMatrix2008} and \citet{boumal2023introductionoptimization}.

We now define the approximate second-order stationarity, which serves as the objective of our methods.
\begin{definition}[Approximate second-order stationary point] \label{def:opt}
	We say $x\in\M$ is an $(\epsilon_g,\epsilon_H)$-approximate second-order stationary point of $f$ if
	$$
		\|\g f(x)\| \le \epsilon_g,\quad
		\lambda_{\min}(\hess f(x)) \ge - \epsilon_H,
	$$
	where $\lambda_{\min}$ returns the smallest eigenvalue of a linear operator on $\tm{x}$.
\end{definition}

Our exploration in this work centers on objective functions with a \holder continuous Hessian.

\begin{definition}[\holder continuity of objective's Hessian] 
	We say $f$'s Hessian is \holder continuous with order $\mu \in (0,1]$, if there exist a constant $C_{H}\ge 0$ such that for any $x,y\in\M$, it holds that
	$$
		\|P_{yx} \hess f(y) P_{xy}  - \hess f(x) \|_{\mathrm{op}} \le C_{H}\dist(x,y)^{\mu},
	$$
	where $\|\cdot \|_{\mathrm{op}}$ is the operator norm of the linear operators on $\tm{x}$.
	We denote $f\in C^{2,\mu}$ if $f\in C^{2}$ and $\hess f$ is $\mu$-order \holder continuous.
\end{definition}

To pull a point on the tangent space back onto the manifold, we use retractions.
Specifically, a map $R: \tm{}\to\M$ is called a retraction if its restriction to any $x\in\M$, denoted as $R_{x}$, satisfies conditions $R_{x}(0) = x$ and $\d R_{x}(0) = \operatorname{id}(\tm{x})$.
The exponential map is a special retraction, and retractions can be viewed as a first-order approximation of the exponential map.
This approximation bias becomes more significant for nonsmooth retractions.
Therefore, prior work often necessitates retractions to be $C^{2}$ or even \textit{second-order} retractions \citep{absilTrustRegionMethodsRiemannian2007,agarwal2021Adaptiveregularization}, whose \textit{acceleration} at the origin is zero: $\d^2 R_{x}(0_{x}) = 0$ (see \citet{absilOptimizationAlgorithmsMatrix2008} or \citet{boumal2023introductionoptimization}).
In contrast, this paper only mandates that retractions' differential be \holder continuous at the origin.
Examples of such retractions can be found in \citet[Remark 1]{zhang2023RiemannianTrust}.

\begin{definition}[\holder continuity of retraction's differential] 
	We say $R_{x}$'s differential is \holder continuous at origin with order $\nu \in (0,1]$, if there exist a constant $C_{R}\ge 0$, such that for any $x\in\M$ and $\xi\in\tm{x}$, it holds that
	$$
		\|P_{yx}\d R_{x}(\xi)P_{xy} - \d R_{x}(0_{x})\|_{\operatorname{op}} \le C_{R}\|\xi\|^{\nu},
	$$
	where $x = R_{x}(0_{x})$ and $y \coloneqq R_{x}(\xi)$.
	We denote $R\in C^{1,\nu}$ if $R\in C^{1}$ and $\d R$ is $\nu$-order \holder continuous.
\end{definition}

The following proposition delineates the accuracy of retractions concerning their approximation to the exponential map.

\begin{proposition}[Retraction properties]\label{lem:ma-diff}
	Suppose $R_{x}$ has a \holder continuous differential with order $\nu\in(0,1]$ and constant $C_{R}$.
	For any $x\in\M$ and $\eta\in\tm{x}$, it holds that
	$$
		\dist (R_x(\eta), \exp_x(\eta)) \le C_{R} \|\eta\|^{1 + \nu},
	$$
	and
	$$
		\|\eta - \exp_{x}^{-1}(R_{x}(\eta))\|
		\le C_{R} \|\eta\|^{1 + \nu}.
	$$
	Moreover, if the operator norm of $\hess f$ is upper bounded by $\beta_{H}$, then the discrepancy between their composition with the objective function is bounded by
	$$
		|f(R_{x}(\eta)) - f(\exp_{x}(\eta))| \le  \left( (1 + C_{R}\|\eta\|^{\nu})\cdot \beta_{H}\|\eta\| + \|\g f(x)\|\right)\cdot C_{R}\|\eta\|^{1+\nu}.
	$$
\end{proposition}

The third inequality presented in \cref{lem:ma-diff} introduces a dependence on $\g f$, a distinctive characteristic of the Riemannian setting involving non-second-order retractions.
As we progress through subsequent sections, it will become evident that this reliance on the gradient holds significant implications for the dynamics of the regularization parameter. Its careful analysis is essential for deriving the bounds of the regularization parameter.

We conclude this section with a blanket assumption upheld throughout the paper, serving as the foundation for all subsequent results.

\begin{assumption} \label{asmp}
	The objective function $f$ is bounded below on the complete Riemannian manifold $\M$.
	The approximation tolerance in \cref{def:opt} satisfies $\epsilon_{g}\vee \epsilon_{H}\le 1$.
	The objective function possesses a \holder continuous Hessian with order $\mu\in (0,1]$ and constant $C_{H}$.
	The retraction has a \holder continuous differential with order $\nu\in(0,1]$ and constant $C_{R}$.
	Finally, the Hessian has a uniformly bounded operator norm on the level set $\mathcal{M}_{f_0}\coloneqq \{ x\in\M: f(x)\le f_0 \}$, i.e., there exists $\beta_{H}>0$ such that $\|\hess f\|_{\mathrm{op}}\le \beta_{H}$ on $\mathcal{M}_{f_0}$.
\end{assumption}

All conditions in \cref{asmp} can be relaxed to apply exclusively along the trajectory of the algorithm.
Moreover, we impose separate conditions on the objective function and the retraction we use, which is more pragmatic and versatile compared to the assumption on their composition $f\circ R$ found in existing literature \citep{boumal2019global,agarwal2021Adaptiveregularization}.

\ifSubfilesClassLoaded{\bibliography{ic}}{}
\section{Riemannian Adaptive Regularized Newton Methods} \label{sec:rarn}
\ifSubfilesClassLoaded{\maketitle}{}

Line search, trust region, and higher-order regularization methods are popular extensions of Newton methods for unconstrained nonconvex optimization.
These diverse methods and their variants can be conceptualized as applying various forms of regularization to the Newton model function, i.e., the truncated second-order Taylor expansion of the objective function:
\begin{equation}\label{eq:model}
	m_x(\eta) \coloneqq f(x) + \left< \eta, \grad f(x) \right> + \frac{1}{2}\left< \eta,\hess f(x)[\eta ] \right>.
\end{equation}
Here, we present a unified Riemannian adaptive regularized Newton (RARN) method that considers the following regularized model problem on a tangent space:
\begin{equation}\label{eq:model-arn}
	\min_{\eta\in\tm{x}}\bar{m}_{x}(\eta;\sigma)\coloneqq m_{x}(\eta) + \phi(\eta;\sigma),
\end{equation}
where $\phi$ is the regularization and $\sigma$ is the adaptive regularization parameter that are dynamically adjusted throughout the iterations.
	{We impose $\phi(0;\sigma_k) = 0$.}
This scheme encompasses adaptive versions of the previously mentioned methods:
\begin{itemize}
	\item For trust region methods, $\phi(\eta;\Delta) = \delta  _{B(0,\Delta)}(\eta)$, where $\delta$ is the indicator function and $B(0,\Delta)$ is the trust region.
	\item For adaptive $p$-order regularization, $\phi(\eta,\sigma ) = \frac{1}{p}\sigma  \|\eta\|^p$.
\end{itemize}

At each iteration of RARN, we (approximately) solve the regularized model problem \cref{eq:model-arn}, and then decide whether to accept the candidate iteration point or not and update the regularization parameter, according to a comparison between the objective decrease and the (unregularized) model decrease.
More precisely, after the subproblem solver returns $\eta_k$, we compute the relative decrease ratio:
\begin{equation}\label{eq:rho}
	\rho_k = \frac{ f(x_k) -  f(R_{x_k}(\eta_k))}{{m}_{x_k}(0) - {m}_{x_k}(\eta_k)}.
\end{equation}
One can also compute the ratio over the regularized model decrease $\bar{m}_{x_k}(0)-\bar{m}_{x_k}(\eta_k)$, while still maintaining the validity of the complexity guarantees presented in this paper, albeit with minor adjustments in the technical details.
We opt for the unregularized model decrease because it is consistent across all RARN variants.
We provide the RARN framework for achieving an approximate second-order stationary point in \cref{alg:arn}.

\begin{algorithm}[!th]
	\caption{Riemannian Adaptive Regularized Newton Method}
	\label{alg:arn}
	{\Input{Initial point $x_0\in\mathcal{M}$, tolorance $\epsilon_{g}, \epsilon_{H} \in (0,1)$, adaptation parameters, control parameters}
		\Output{Last iterate $x_k$}}
	\For{$k=0,1,\dots$}{
		\uIf (\tcp*[f]{first-order stationarity test}) {$\|\grad f(x_k)\| > \epsilon_{g}$} {
			obtain $\eta_k \approx \arg\min_{\eta\in\tm{x}}\bar{m}_{x_k}(\eta;\sigma_k)$ with termination criterion~(\ref{eq:tc-1}\tcand\hyperref[eq:tc-c]{C})
			\label{line:1}
		} \uElseIf (\tcp*[f]{second-order stationarity test}) {$\hess f(x_k) \not\succeq -\epsilon_{H}I$} {
			obtain $\eta_k \approx \arg\min_{\eta\in\tm{x}}\bar{m}_{x_k}(\eta;\sigma_k)$ with termination criterion~(\ref{eq:tc-2}\tcand\hyperref[eq:tc-c]{C})
			\label{line:2}
		} \Else {
			\Return $x_k$
		}
		compute $\rho_k$ using \eqref{eq:rho}\;
		shrink or expand $\sigma_{k+1}$ according to $\rho_k$\;
		accept $x_{k+1} = R_{x_k}(\eta_k)$ or not according to $\rho_k$\;
	}
\end{algorithm}

Within the uniform algorithm framework, we also introduce termination criteria that offer a greater degree of inexactness and flexibility compared to prior work \citep{cartis2011Adaptivecubic,agarwal2021Adaptiveregularization,curtis2021TrustRegionNewtonCG,xu2020Newtontypemethods}.
First of all, for all iterations, we require the candidate point to outperform the Cauchy point, which is the solution to \cref{eq:model-arn} restricted along the single dimension spanned by $g_k$:
$$
	\eta_{k}^{\mathrm{C}} = \argmin_{\eta\in \spa \{ g_k \}} \bar{m}_{x_k}(\eta).
$$
This leads to the Cauchy termination criterion:
\begin{equation}\label{eq:tc-c}\tag{TC.C}
	\bar{m}_{x_k}(\eta_k) \le \bar{m}_{x_k}(\eta_{k}^{\mathrm{C}}).
\end{equation}
If a candidate point satisfies \cref{eq:tc-c}, by the definition of the regularized model problem \cref{eq:model-arn}, we have
\begin{equation}\label{eq:cauchy}
	m_{x_k}(0) - m_{x_k}(\eta_k)
	= \bar{m}_{x_k}(0) - \bar{m}_{x_k}(\eta_k) + \varphi(\eta_k;\sigma_k)
	\ge \varphi(\eta_k;\sigma_k).
\end{equation}

If the first-order test fails, i.e., $\|g_k\|>\epsilon_{g}$, the subproblem solver needs to minimize the residual $\grad \bar{m}(\eta)$. Thus, we employ the following first-order termination criterion:
\begin{equation}\label{eq:tc-1}\tag{TC.1}
	\|\grad \bar{m}_{x_k}(\eta_k)\| \le \|\eta_k\|^{1+\theta_1},
\end{equation}
where $\theta_1\in(0,1]$ is the parameter governing the degree of inexactness of the subproblem solution.
We remark that due to the non-differentiability of the indicator function, we incorporate the trust region as a problem constraint rather than integrating it into the regularized model function $\bar{m}$. We defer the comprehensive definition of the model problem of RTR to \cref{sec:rtr}.

If the second-order stationarity is targeted and the second-order test fails, i.e., $H_k \not\succeq -\epsilon_{H}I$, the subproblem solver should aim for the following criterion:
\begin{equation}\label{eq:tc-2}\tag{TC.2}
	\hess \bar{m}_{x_k}(\eta_k) \succeq -\|\eta_k\|^{\theta_2}\cdot I,
\end{equation}
where $\theta_2\in(0,1]$ is an inexactness parameter.
For simplicity, we use the same parameter $\theta=\theta_1=\theta_2$ to control the subproblem solver's inexactness in our analysis.
The above termination conditions only serve as basic criteria for our unified framework.
In specific methods, the subproblem solver can incorporate alternative suitable termination conditions, such as truncation conditions for a trust region method.
These conditions also directly establish lower bounds on the step-size $\|\eta_k\|$, subsequently forming lower bounds for the successful objective decrease, as will be demonstrated in the forthcoming sections.

\begin{proposition}[Termination criteria]\label{lem:tc}
	{\upshape 1)} For $\eta_k$ satisfying \cref{eq:tc-1}, if $k$th iteration is successful and $\|\eta_k\|\le 1$, we have 
	\begin{equation*}
		\|g_{k+1}\| \le (2C_{H}(C_{R}+1)+ \beta_H C_{R} + 1)\|\eta_k\|^{1+\mu\w\nu\w\theta} + \|\g \phi(\eta_k;\sigma_k)\|.
	\end{equation*}
	{\upshape 2)} For $\eta_k$ satisfying \cref{eq:tc-2}, we have 
	$$
		-\lambda_{\min}(H_k) \le \|\eta_k\|^{\theta} + \lambda_{\max}(\hess \phi(\eta_k;\sigma_k)).
	$$
\end{proposition}

Readers may notice that the equivalence of $\mu$, $\nu$, and $\theta$ has emerged in the first inequality of \cref{lem:tc}.

While \cref{alg:arn} provides a high-level framework, we will specify the subproblem solver, additional termination criteria, regularization parameter adaptation, and acceptance criteria in dedicated sections for trust region and adaptive regularization methods.

\subsection{Iteration Complexity Analysis Framework}\label{sec:ic-frame}
We now give a unified analysis framework for the iteration complexity of RARN.
A key observation is that, in RARN, the adaptation of the regularization parameter is closely related to the acceptance of a candidate step.
We categorize the $k$th iteration as either ``successful'' or ``unsuccessful'' based on whether $R_{x_k}(\eta_k)$ is accepted as $x_{k+1}$.
Formally, we define the following index sets:
$$
	\begin{gathered}
		\mathcal{K} \coloneqq \{ k\in\mathbb{N}: \text{the algorithm is not terminated within }k\text{th iteration} \},\\
		\mathcal{S} \coloneqq \{ k\in \mathcal{K}: k\text{th iteration is successful} \}, \ \ \text{and}\ \
		\mathcal{U} \coloneqq \{ k\in \mathcal{K}: k\text{th iteration is unsuccessful} \}.
	\end{gathered}
$$
Because $|\mathcal{K}| = |\mathcal{S} \sqcup \mathcal{U}| = |\mathcal{S}| + |\mathcal{U}|$, we can establish a connection between the total number of iterations and the range of values for the regularization parameter. We validate this observation in the following lemma, an adaptation from \citet[Theorem 2.1]{cartis2011Adaptivecubic}.

\begin{lemma}[Decomposition of total number of iterations]\label{lem:K}
	Suppose that \cref{alg:arn} prescribes the following conditions for the regularization parameter: $\sigma_k$ only shrinks for $k\in\mathcal{S}$, expands for any $k\in \mathcal{U}$, and the shrinkage is explicitly lower bounded by $\sigma_{k+1} \ge \max \{ \ubar{\kappa}\sigma_k,\ubar{\sigma} \}$ for some $\ubar{\kappa}<1$, while the expansion is upper bounded by $\sigma_{k+1}\le \bar{\kappa}\sigma_{k}$ for some $\bar{\kappa} >1$.
	Additionally, if
	\begin{enumerate*}[label=\upshape\arabic*\upshape)]
		\item $\mathcal{S}$ is finite and \item a upper bound $\bar{\sigma}<+\infty$ of $\sigma_k$ exists,
	\end{enumerate*}
	then
	$$
		|\mathcal{K}| \le \left( 1 - \frac{\log \ubar{\kappa}}{\log\bar{\kappa}} \right)|\mathcal{S}| + \log_{\bar{\kappa}}\frac{\bar{\sigma}}{\ubar{\sigma}}.
	$$

	Similarly, suppose \cref{alg:arn} specifies that $\sigma_k$ only expands for $k\in\mathcal{S}$, shrinks for any $k\in \mathcal{U}$, and the expansion is explicitly upper bounded by $\sigma_{k+1} \le \min \{ \bar{\kappa}\sigma_k,\bar{\sigma} \}$ for some $\bar{\kappa}>1$, while the shrinkage is lower bounded by $\sigma_{k+1}\ge \ubar{\kappa}\sigma_{k}$ for some $\ubar{\kappa} <1$.
	Additionally, if
	\begin{enumerate*}[label=\upshape\arabic*\upshape)]
		\item $\mathcal{S}$ is finite and \item a lower bound $\ubar{\sigma}>0$ of $\sigma_k$ exists,
	\end{enumerate*}
	then 
	$$
		|\mathcal{K}| \le \left( 1 - \frac{\log \bar{\kappa}}{\log\ubar{\kappa}} \right)|\mathcal{S}| + \log_{\ubar{\kappa}^{-1}}\frac{\bar{\sigma}}{\ubar{\sigma}}.
	$$
\end{lemma}
\begin{proof}
	We prove the first bound; the second bound can be derived similarly.
	For any $k\in \mathbb{N}$, by the assumptions, we have
	$$
		\bar{\sigma} \ge \sigma_k \ge \sigma_0 \cdot \ubar{\kappa}^{|\mathcal{S}\cap [k]|} \cdot \bar{\kappa}^{|\mathcal{U}\cap [k]|} \ge \ubar{\sigma} \cdot \ubar{\kappa}^{|\mathcal{S}\cap [k]|} \cdot \bar{\kappa}^{|\mathcal{U}\cap [k]|}.
	$$
	Taking the logarithm of both sides and rearranging the above inequality gives
	$$
		|\mathcal{U}\cap [k]| \le -\frac{\log \ubar{\kappa}}{\log \bar{\kappa}}\cdot |\mathcal{S}\cap [k]| + \log_{\bar{\kappa}} \frac{\bar{\sigma}}{\ubar{\sigma}}
		\le -\frac{\log \ubar{\kappa}}{\log \bar{\kappa}}\cdot |\mathcal{S}| + \log_{\bar{\kappa}} \frac{\bar{\sigma}}{\ubar{\sigma}},
	$$
	We get the finite bound of $|\mathcal{U}|$ by letting $k\to\infty$, and then the result follows.
\end{proof}

\begin{remark}
	\cref{lem:K} directly gives a nonasymptotic bound:
	$
		|\mathcal{K}| = O\left(|\mathcal{S}| + \log (\bar{\sigma} /\ubar{\sigma})\right).
	$
\end{remark}

To explicitly bound $|\mathcal{K}|$, we still need to validate the two assumptions in \cref{lem:K}: providing a bound of $|\mathcal{S}|$ and an upper/lower bound of $\sigma_k$.
For the bound of $|\mathcal{S}|$, if we can establish a minimal decrease in the objective function during successful iterations, then since $f$ is bounded below, the number of successful iterations can be bounded.

To this end, we further partition $\mathcal{S}$ into the following index sets
$$
	\begin{array}{ll}
		\mathcal{S}_{L}\coloneqq \{ k\in \mathcal{S}: \|\grad f(x_k)\| \le \epsilon_{g} \},\quad
		\mathcal{S}_{G}\coloneqq \{ k\in \mathcal{S}: \|\grad f(x_k)\| > \epsilon_{g} \},                                                                      \\
		\mathcal{S}_{GL}\coloneqq  \{ k\in \mathcal{S}: \|\grad f(x_k)\| > \epsilon_{g} \text{ and } \|\grad f(x_{k+1})\| \le \epsilon_{g} \},\quad \text{and} \\
		\mathcal{S}_{GG}\coloneqq  \{ k\in \mathcal{S}: \|\grad f(x_k)\| > \epsilon_{g} \text{ and } \|\grad f(x_{k+1})\| > \epsilon_{g} \}.                   \\
	\end{array}
$$
We list some observations concerning the above index sets. First, $\mathcal{S} = \mathcal{S}_{L}\sqcup \mathcal{S}_{G}$ and $\mathcal{S}_{G} = \mathcal{S}_{GL}\sqcup \mathcal{S}_{GG}$. Moreover, if only the first-order stationarity is targeted, $|\mathcal{S}| = |\mathcal{S}_{GG}| + 1$. Also, for $k\in \mathcal{S}_{GL}$, any immediate successful iteration following $k$ (if exists) must fall into the index set $\mathcal{S}_{L}$, implying that $|\mathcal{S}_{GL}|\le |\mathcal{S}_{L}| + 1$.
Furthermore, $\eta_k$ is returned by \cref{line:1} of \cref{alg:arn} if $k\in \mathcal{S}_{G}$ and by \cref{line:2} if $k\in \mathcal{S}_{L}$.

\begin{lemma}[Number of successful iterations]\label{lem:S}
	For $k\in \mathcal{S}$, if \cref{line:1,line:2} in \cref{alg:arn} produce a minimal decrease of $\epsilon_1$ and $\epsilon_2$ in $f$ respectively, we have
	$$
		|\mathcal{S}| \le 1+ \begin{cases}
			(f_0 - f_{\min})\cdot \epsilon_{1}^{-1},\quad                 & \textup{for first-order stationarity,}  \\
			2(f_0 - f_{\min})\cdot (\epsilon_{1}\w \epsilon_2)^{-1},\quad & \textup{for second-order stationarity.}
		\end{cases}
	$$
\end{lemma}
\begin{proof}
	Since the objective function value only changes when the iteration is successful, we have
	$$
		\begin{aligned}
			f_0 - f_{\min} & \ge \sum_{k\in \mathcal{S}} (f_{k}-f_{k+1})
			\ge \sum_{k\in \mathcal{S}_{GG}} (f_{k} - f_{k+1}) + \sum_{k\in \mathcal{S}_{L}} (f_{k} - f_{k+1}) \\
			               & \ge |\mathcal{S}_{GG}|\cdot \epsilon_1 + |\mathcal{S}_{L}|\cdot \epsilon_2
			\ge (|\mathcal{S}_{GG}| + |\mathcal{S}_{L}|) \cdot (\epsilon_1 \w \epsilon_2).
		\end{aligned}
	$$
	For first-order stationarity, $|\mathcal{S}_{L}|=0$. Thus, we have
	\[
		|\mathcal{S}| \le |\mathcal{S}_{GG}| + 1 \le 1 + (f_0-f_{\min})\cdot \epsilon_{1}^{-1}
		.\]
	For second-order stationarity, any $k$ in $\mathcal{S}_{GL}$ is followed by a successful iteration in $\mathcal{S}_{L}$. Thus, $|\mathcal{S}_{GL}| \le |\mathcal{S}_{L}|$. Then, we have
	$$
		|\mathcal{S}| \le |\mathcal{S}_{GG}| + 2|\mathcal{S}_{L}| + 1 \le 1+ 2(f_0-f_{\min})\cdot (\epsilon_{1}\w \epsilon_2)^{-1}
		.$$
\end{proof}

By \cref{lem:K,lem:S}, we get an iteration complexity guarantee once we determine the minimal successful decreases $\epsilon_1$ and $\epsilon_2$, as well as the regularization parameter bound $\bar{\sigma}$/$\ubar{\sigma}$ for specific algorithms.
In the following sections, our objective is to concretize these values.

\begin{corollary}[Total number of iterations]\label{cor:K}
	If the conditions in \cref{lem:K,lem:S} are satisfied, we have
	$$
		|\mathcal{K}| \le \begin{cases}
			O(\epsilon_1^{-1} + \log \bar{\sigma} /\ubar{\sigma}), \quad               & \textup{for first-order stationarity},  \\
			O((\epsilon_1\w\epsilon_2)^{-1} + \log \bar{\sigma} /\ubar{\sigma}), \quad & \textup{for second-order stationarity}.
		\end{cases}
	$$
\end{corollary}

\ifSubfilesClassLoaded{\bibliography{ic}}{}
\section{Riemannian Adaptive \texorpdfstring{$2\!+\!\alpha$}{2+alpha} Regularization} \label{sec:rar}

\ifSubfilesClassLoaded{\maketitle}{}

Adaptive regularization with cubics (ARC) \citep{cartis2011Adaptivecubica,agarwal2021Adaptiveregularization} adds a cubic regularization term to the vanilla quadratic model function \cref{eq:model}.
However, as we will shortly elucidate, for $C^{2,\mu}$ objective functions and $C^{1,\nu}$ retractions, cubics \textit{under-regularize}.
Therefore, for the time being, we first consider a general Riemannian adaptive $\ar{\omega}$ regularization (RAR), $\omega\in(0,1]$.
Specifically, the Riemannian $\ar{\omega}$ regularized model problem at $x_k\in\M$ reads:
\begin{align}\label{eq:model-ar}
	\min_{\eta \in T_{x_k}\mathcal{M}} \quad & \bar{m}^{\mathrm{ar}}_{x_k}(\eta;\sigma_k) \coloneqq f(x_k) + \left< \eta, \grad f(x_k) \right> + \frac{1}{2} \langle \eta,\hess f(x_k)[\eta]\rangle + \frac{\sigma_k}{2+\omega}\|\eta\|^{2+\omega}.
\end{align}
{We inherit the method framework in \cref{alg:arn}
and solve the subproblem with termination criteria \cref{eq:tc-c}, \cref{eq:tc-1}, and \cref{eq:tc-2}.
The method is presented in Algorithm~\ref{alg:ar}.
We require the subproblem solver to enjoy the following property.}

\begin{assumption}[Gradient bound by step size]\label{lem:g-bound}
	In \cref{alg:ar}, the subproblem solver for approximately solving  \eqref{eq:model-ar} returns a vector $\eta_k$ satisfying
	$$\|g_k\|\le \beta_{H}\|\eta_k\| + \|\g \phi(\eta_k;\sigma_k)\|.$$
\end{assumption}

{Specifically, we propose employing a Lanczos-based Krylov subspace method as the subproblem solver {(see \cref{sec:kry} and \cite{cartis2011Adaptivecubica,agarwal2021Adaptiveregularization,carmon2020FirstOrderMethods})}.
Krylov subspace methods satisfy \cref{lem:g-bound} (verified in \cref{sec:pf-g-bound}).
Similar to \cref{asmp}, we assume \cref{lem:g-bound} holds for all subproblem solvers employed.}

\begin{algorithm}[!th]
	\caption{Riemannian Adaptive $\ar{\omega}$ Regularization}
	\label{alg:ar}
	{
	\Input{Initial point $x_0\in\mathcal{M}$, tolorance $\epsilon_{g}, \epsilon_{H} \in (0,1)$, {regularization} adaptation parameters $\kappa_3>\kappa_2 \ge 1 > \kappa_1 > 0$, $\ubar{\sigma} > 0, \sigma_0 \in [\ubar{\sigma}, \infty)$, control parameters $1> \varrho_2\ge\varrho_1 > 0$}
	\Output{Last iterate $x_k$}
	}
	\For{$k=0,1,\dots$}{
		\uIf {$\|g_k\| > \epsilon_{g}$} { \label{line:ar_sub_beg}
			obtain $\eta_k$ by solving problem \eqref{eq:model-ar} with termination criteria~(\ref{eq:tc-1}\tcand\hyperref[eq:tc-c]{C})\;
		} \uElseIf {$H_k\not\succeq -\epsilon_{H}I$} {
			obtain $\eta_k$ by solving problem \eqref{eq:model-ar} with termination criteria~(\ref{eq:tc-2}\tcand\hyperref[eq:tc-c]{C})\;
		} \Else {
			\Return $x_k$\;
		}\label{line:ar_sub_end}

		compute $\rho_k$ using \eqref{eq:rho}\;
		set $\sigma_{k+1} = \left\{\begin{aligned}&[\max \{ \ubar{\sigma}, \kappa_1\sigma_k \}, \sigma_k], &&\rho_k > \varrho_2, && \texttt{// very successful}\\ &[\sigma_k, \kappa_2\sigma_k], && \varrho_1 \le \rho_k \le \varrho_2, && \texttt{// successful}\\ &[\kappa_2\sigma_k, \kappa_3\sigma_k], && \text{o.w.} && \texttt{// unsuccessful}\end{aligned}\right.$\;
		set $x_{k+1} = \begin{cases}R_{x_k}(\eta_k), &\rho_k \ge \varrho_1,\\  x_k & \text{o.w.}\end{cases}$
	}
\end{algorithm}

\subsection{Iteration Complexity of RAR}

As discussed in \cref{sec:ic-frame}, obtaining an iteration complexity bound of RAR only requires establishing an upper bound of $\sigma_k$ and the minimal decrease in $f$ for $k\in \mathcal{S}$.
When using a second-order retraction and $\omega = \mu$, one can readily derive an upper bound of $\sigma_k$, akin to the Euclidean case, as demonstrated in \cite[Lemma 2.2]{birgin2017Worstcaseevaluation} and \cite[Lemma 5.2]{cartis2011Adaptivecubica}.
However, a general retraction introduces a gradient dependency in the model-actual difference (see \cref{lem:ma-diff} or \cite[Assumption 4]{agarwal2021Adaptiveregularization}).
Notably, \citet{agarwal2021Adaptiveregularization} introduced an additional smoothness assumption on the composition $f\circ R$ (\cite[Assumption 2]{agarwal2021Adaptiveregularization}) to bypass this challenge. We will show that this additional assumption is unnecessary.
Additionally, the general order $\omega$ introduces a step-size dependency in $\sigma_k$, which vanishes when $\omega = \mu\w\nu$ (see \eqref{eq:u-2}).
These two dependencies cause $\sigma_k$ much more challenging to bound for RAR with $\ar{\omega}$ regularization.
To address the step-size dependency, we will utilize the following proposition.

\begin{proposition}[Step-size bound]\label{prop:s}
	If the Cauchy condition \cref{eq:tc-c} is satisfied, we have
	$$
		\|\eta_k\| \le \left(\frac{3(\beta_{H}+1)}{\sigma_k}\right)^{1 /\omega} \vee \left( \|g_k\| \w \left( \frac{6\|g_k\|}{\sigma_k}\right)^{1 /(1+\omega)}  \right).
	$$
\end{proposition}

The proposition is a direct consequence of the Cauchy condition, and we defer its proof to \cref{sec:pf-s}.
Furthermore, we will show that the step-size has a lower bound determined by $\epsilon_{g}$ and $\epsilon_{H}$.
To address the gradient dependency, in addition to \cref{lem:g-bound}, we also need the following proposition.

\begin{proposition}[Gradient bound]\label{prop:g}
	The gradient sequence generated by \cref{alg:ar} is bounded, i.e., there exists $\beta_{g}>0$ such that $\|g_k\|\le \beta_{g}$ for any $k\in\mathcal{K}$.
\end{proposition}

In practice, there are various means to bound the gradient sequence. For instance, if the level set $\{ x\in\M:f(x)\le f_0 \}$ is compact or if the algorithm converges, $\{ g_k \}$ is bounded. We provide a proof of \cref{prop:g} adapted from \cite[Theorem 2.5]{cartis2011Adaptivecubica} in \cref{sec:pf-g-bound-uni} without any additional assumptions.
We are now ready to provide an upper bound of the regularization parameter.

\begin{lemma}[Regularization parameter upper bound]\label{lem:sigmabar}
	The regularization parameter has an upper bound
	$$
		\bar{\sigma}\coloneqq C_{\sigma}\left(\epsilon_{g}\w \epsilon_{H}^{1 /\alpha}\right)^{-[\omega-\alpha]_{+}},
	$$
	where $C_{\sigma}$ is a positive constant satisfying $C_\sigma\ge 1$ and $[x]_+ \coloneqq \max \{ 0,x \}$.
\end{lemma}

\begin{proof}
	In this proof, we omit the superscript of $\bar{m}^{\mathrm{ar}}_{x_k}$ and the subscript $x_k$ of $\bar{m}_{x_k}$, $m_{x_k}$, and $R_{x_k}$.
	According to the update rule in \cref{alg:ar}, the regularization parameter only expands when $\rho_k\le\varrho_{2}$. Therefore, we only need to derive the upper bound of $\sigma_k$ when $\rho_k \le \varrho_2$, in which case \eqref{eq:rho} gives
	$$
		\begin{aligned}
			\varrho_2\left({m}(0)-{m}\left(\eta_k\right)\right)
			 & \ge f\left(x_k\right)- f\left(R\left(\eta_k\right)\right)      \\
			 & ={m}(0)-m\left(\eta_k\right)
			+ m\left(\eta_k\right)- f\left(\exp \left(\eta_k\right)\right)
			+ f\left(\exp (\eta_k)\right)- f\left(R\left(\eta_k\right)\right) \\
			 & \geqslant {m}(0)-m\left(\eta_k\right)-
			|m(\eta_k) - f(\exp (\eta_k))| - |f(\exp (\eta_k)) - f(R(\eta_k))|,
		\end{aligned}
	$$
	which can be reformulated as
	\begin{equation}\label{eq:oa-diff-1}
		(1-\rh_2) \left(\bar{m}(0)-\bar{m}\left(\eta_k\right) + \phi(\eta_k;\sigma_k)\right) \le |f(R(\eta_k)) - f(\exp (\eta_k))| + |f(\exp (\eta_k)) - m(\eta_k)|.
	\end{equation}
	The first term in the right hand side of \cref{eq:oa-diff-1} can be bounded by \cref{lem:ma-diff}.
	For the second term, by the Taylor expansion on manifolds (see, e.g., \cite[Lemma 3]{zhang2023RiemannianTrust}), there exists $\tau\in[0,1]$ such that
	$$
		|f(\exp (\eta_k)) - m(\eta_k)| = \left| \frac{1}{2}\left<\eta_k, (H_k - P_{\gamma}^{\tau\to 0}H_{\tau}P_{\gamma}^{0\to\tau})\eta_k \right> \right|
		\le \frac{1}{2}\|H_k -P_{\gamma}^{\tau\to 0}H_{\tau}P_{\gamma}^{0\to\tau}\|\|\eta_k\|^2.
	$$
	where $\gamma$ is the geodesic from $x_k$ to $\exp (\eta_k)$.
	Then by the \holder continuity of $\hess f$, we get
	\begin{equation}\label{eq:oa-diff-2}
		|f(\exp (\eta_k)) - m(\eta_k)| \le \frac{1}{2}C_{H}\tau\|\eta_k\|^{2+\mu} \le C_{H}\|\eta_k\|^{2+\mu}.
	\end{equation}
	Combining \cref{eq:oa-diff-1}, \cref{eq:oa-diff-2}, the Cauchy condition \cref{eq:cauchy}, and \cref{lem:ma-diff} gives
	\begin{equation}\label{eq:u}
		(1-\rh_2) \phi(\eta_k;\sigma_k)
		\leqslant C_H\left\|\eta_k\right\|^{2+\mu}+(1+C_{R}\|\eta_k\|^{\nu})\cdot \beta_{H}C_{R}\left\|\eta_k\right\|^{2+\nu} + C_{R}\|g_k\|\|\eta_k\|^{1+\nu}.
	\end{equation}
	A general $\ar{\omega}$ regularization necessitates the consideration of two cases: small and large step-sizes.\\
	\case{\rom{1}} When $\|\eta_k\|\le C_{\sigma,1}\coloneqq 1\w (\frac{1-\rh_2}{2(2+\omega)C_{R}})^{1/\nu}$, by \cref{lem:g-bound}, we have
	$$
		\frac{1-\rh_2}{2+\omega}\sigma_k\|\eta_k\|^{2+\omega} \le (C_{H}+(1+C_{R})\beta_{H}C_{R})\|\eta_k\|^{2+\mu\w\nu} + C_{R}(\beta_{H} + \sigma_k\|\eta_k\|^{\omega})\|\eta_k\|^{2+\nu},
	$$
	which gives
	$$
		\frac{1-\rh_2}{2(2+\omega)}\sigma_k\|\eta_k\|^{2+\omega} \le (C_{H}+(2+C_{R})\beta_{H}C_{R})\|\eta_k\|^{2+\mu\w\nu}.
	$$
	Let $C_{\sigma,2}\coloneqq 2(2+\omega)(C_{H}+(2+C_{R})\beta_{H}C_{R})/(1-\rh_2)$. The above inequality is equivalent to
	\begin{equation}\label{eq:u-2}
		\sigma_k\|\eta_k\|^{\omega} \le C_{\sigma,2}\|\eta_k\|^{\mu\w\nu}.
	\end{equation}
	\case{\rom{1}.\rom{1}} If $\|g_k\|>\epsilon_{g}$, by \cref{lem:g-bound,eq:u-2}, we get
	$$
		\epsilon_{g}\le (\beta_{H}+\sigma_k\|\eta_k\|^{\omega})\|\eta_k\|
		\le (\beta_{H}+C_{\sigma,2}\|\eta_k\|^{\mu\w\nu})\|\eta_k\| \le (\beta_{H}+C_{\sigma,2})\|\eta_k\|,
	$$
	which gives $\|\eta_k\|\ge \epsilon_{g}/(\beta_{H}+C_{\sigma,2})$.\\
	\case{\rom{1}.\rom{2}} If $\|g_k\|\le \epsilon_{g}$, termination condition \cref{eq:tc-2} is used. Then if the algorithm dose not terminate, by \cref{lem:tc} and \cref{eq:u-2}, we get
	$$
		\begin{aligned}
			\epsilon_{H}
			\le & \|\eta_k\|^{\theta} + \lambda_{\max}\left(\hess \phi(\eta_k;\sigma_k)\right)                                                                    \\
			=   & \|\eta_k\|^{\theta} + \lambda_{\max}\left( \sigma_k\|\eta_k\|^{\omega}\left( \frac{\eta_{k}\eta_{k}^T}{\omega\|\eta_k\|^2} + I \right)\ \right) \\
			\le & \|\eta_k\|^{\theta} + \frac{1+\omega}{\omega}\sigma_k\|\eta_k\|^{\omega}                                                                        \\
			\le & (1+ C_{\sigma,2}(1+\omega)/\omega)\|\eta_k\|^{\mu\w\nu\w\theta},
		\end{aligned}
	$$
	which gives $\|\eta_k\|\ge (\frac{\epsilon_{H}}{1+C_{\sigma,2}(1+\omega)/\omega})^{1 /\alpha}$.
	Let $C_{\sigma,3} \coloneqq (\beta_{H}+C_{\sigma,2}) \vee (1+C_{\sigma,2}(1+\omega)/\omega)^{1 /\alpha}$.
	For \case{\rom{1}}, we have $\|\eta_k\| \ge C_{\sigma,3}^{-1} (\epsilon_{g}\w \epsilon_{H}^{1/\alpha})$.
	Plugging it back into \cref{eq:u-2} gives
	$$
		\sigma_k \le C_{\sigma,2}\|\eta_k\|^{\mu\w\nu - \omega}
		\le C_{\sigma,2}\|\eta_k\|^{\alpha - \omega}
		\le C_{\sigma,2}\|\eta_k\|^{-[\omega-\alpha]_{+}} \le C_{\sigma,4}(\epsilon_{g}\w\epsilon_{H}^{1 /\alpha})^{-[\omega-\alpha]_{+}},
	$$
	where $C_{\sigma,4}\coloneqq C_{\sigma,2}C_{\sigma,3}^{[\omega-\alpha]_{+}}$.

	\case{\rom{2}} When $\|\eta_k\| > C_{\sigma,1}$, \cref{prop:s} introduces another two cases.\\
	\case{\rom{2}.\rom{1}} When the former term in \cref{prop:s} is active, we have
	$$
		\sigma_k \le \|\eta_k\|^{-\omega}\cdot 3(\beta_{H}+1) \le C_{\sigma,5}\coloneqq 3C_{\sigma,1}^{-\omega}(\beta_{H}+1).
	$$
	\case{\rom{2}.\rom{2}} When the latter term is active in \cref{prop:s}, \cref{eq:u} and \cref{prop:g} give
	$$
		\frac{1-\rh_2}{2+\omega}\sigma_k\|\eta_k\|^{2+\omega} \le
		C_H\beta_g^{2+\mu}+(1+C_{R}\beta_g^{\nu})\cdot \beta_{H}C_{R}\beta_g^{2+\nu} + C_{R}\beta_{g}\beta_g^{1+\nu},
	$$
	which further gives
	$$
		\sigma_k \le C_{\sigma,6} \coloneqq \frac{2+\omega}{(1-\rh_2) C_{\sigma,1}^{2+\omega}} \left( C_{H}\beta_{g}^{2+\mu}+C_{R}\beta_{g}^{2+\nu}+(1+C_{R}\beta_{g}^{\nu})C_{R}\beta_{H}\beta_{g}^{2+\nu} \right).
	$$
	Therefore, for \case{\rom{2}}, we have $\sigma_k\le C_{\sigma,5}\vee C_{\sigma,6}$.
	Combining \case{\rom{1}} and \case{\rom{2}} gives
	$$
		\sigma_k\le (C_{\sigma,4}\vee C_{\sigma,5}\vee C_{\sigma,6})\cdot (\epsilon_{g}\w \epsilon_{H}^{1 /\alpha})^{-[\omega-\alpha]_{+}}.
	$$
	Let $C_{\sigma}\coloneqq \kappa_3(C_{\sigma,4}\vee C_{\sigma,5}\vee C_{\sigma,6})$. We conclude $\sigma_k$ has an upper bound:
	$\bar{\sigma}\coloneqq C_{\sigma}(\epsilon_{g}\w \epsilon_{H}^{1 /\alpha})^{-[\omega-\alpha]_{+}}.$
\end{proof}

As we can see, when $\omega > \alpha$, the algorithm \textit{under-regularizes}, meaning it provides inadequate regularization compared to methods with smaller $\omega$. This can result in potentially larger regularization parameters, especially when the tolerance $\epsilon_{g}$ or $\epsilon_{H}$ is set to a small value. The underlying reason for this behavior is that in cases where the problem exhibits limited smoothness, the quadratic model function $m_{x}$ may fail to deliver a sufficiently accurate approximation of the objective function, particularly at points farther away from $x$. As a remedy, the regularization parameter is increased to constrain the step size.

\begin{lemma}[Minimal successful decrease of RAR]\label{lem:obj-dec-ar}
	For $k\in \mathcal{S}$, we have
	$$
		f(x_k) - f(x_{k+1}) \ge \begin{cases}
			C_s(\epsilon_{g} / \bar{\sigma})^{\tfrac{2+\omega}{1+\omega}},
			\quad & k\in \mathcal{S}_{GG}, \\[2ex]
			C_s(\epsilon_{H} / \bar{\sigma})^{\tfrac{2+\omega}{\omega}},
			\quad & k\in \mathcal{S}_{L},
		\end{cases}
	$$
	where $C_s$ is a positive constant satisfying $C_s\ge 1$ and $\bar{\sigma}$ is specified in \cref{lem:sigmabar}.
\end{lemma}
\begin{proof}
	By the Cauchy condition \cref{eq:cauchy}, for $k\in \mathcal{S}$, we have
	\begin{equation}\label{eq:s}
		f(x_k) - f(x_{k+1}) \ge \rh_1(m_{x_k}(0) - m_{x_k}(\eta_k))
		\ge \frac{\rh_1 \ubar{\sigma}}{2+\omega}\|\eta_k\|^{2+\omega}.
	\end{equation}
	When $\|\eta_k\| \ge 1$, \cref{eq:s} directly satisfies the requirement because
	$$
		\frac{\rh_1\ubar{\sigma}}{(2+\omega)}\|\eta_k\|^{2+\omega} \ge \frac{\rh_1\ubar{\sigma}}{(2+\omega)}
		\gtrsim 1 \gtrsim  (\epsilon_{g} / \bar{\sigma})^{1 / (1+\omega)} \vee (\epsilon_{H} /\bar{\sigma})^{1 /\omega}.
	$$ Thus, we only consider the case where $\|\eta_k\|<1$.
	For $k\in \mathcal{S}_{GG}$, by \cref{lem:tc}, we get
	\begin{equation}\label{eq:ar-obj-dec-1}
		\epsilon_{g}\le C_{s,1}\|\eta_k\|^{1+\alpha} + \bar{\sigma}\|\eta_k\|^{1+\omega},
	\end{equation}
	where $C_{s,1} \coloneqq 2C_{H}(C_{R}+1)+\beta_{H}C_{R}+1$ and $\bar{\sigma}$ is specified in \cref{lem:sigmabar}.
	We claim that the right hand side of \cref{eq:ar-obj-dec-1} is always dominated by $\bar{\sigma}\|\eta_k\|^{1+\omega}$, i.e., $\|\eta_k\|^{1+\alpha}\lesssim \bar{\sigma}\|\eta_k\|^{1+\omega}$. Since $\|\eta_k\|<1$, this is obviously true when $\alpha\ge \omega$.
	Now suppose $\alpha<\omega$ and the right hand side of \cref{eq:ar-obj-dec-1} is dominated by $C_{s,1}\|\eta_k\|^{1+\alpha}$.
	Then, \cref{eq:ar-obj-dec-1} gives
	$$
		\|\eta_k\| \gtrsim \epsilon_{g}^{1 /(1+\alpha)} \ge \epsilon_{g} \ge \epsilon_{g}\w \epsilon_{H}^{1 /\alpha}.
	$$
	Since $\alpha < \omega$, by \cref{lem:sigmabar}, we get
	$$
		\|\eta_k\|^{\alpha - \omega} \lesssim \left( \epsilon_{g}\w \epsilon_{H}^{1 /\alpha} \right)^{\alpha-\omega} \lesssim \bar{\sigma},
	$$
	which indicates $\|\eta_k\|^{\alpha} \lesssim \bar{\sigma}\|\eta_k\|^{\omega}$, affirming the dominance of $\bar{\sigma}\|\eta_k\|^{\omega}$.
	Therefore, we conclude that there exists $C_{s,2} > 0$ such that \cref{eq:ar-obj-dec-1} is equivalent to
	$$
		\epsilon_{g} \le C_{s,2}\bar{\sigma}\|\eta_k\|^{1+\omega},
	$$
	which further gives $\|\eta_k\| \ge (\epsilon_{g}/(\bar{\sigma}C_{s,2}))^{1 /(1+\omega)}$.

	For $k\in \mathcal{S}_{L}$, by \cref{lem:tc}, we have
	$$
		\epsilon_{H} \le \|\eta_k\|^{\theta} + \bar{\sigma}\|\eta_k\|^{\omega}.
	$$
	Similarly, we claim that the right hand side of the above inequality is dominated by $\bar{\sigma}\|\eta_k\|^{\omega}$. Therefore, there exists $C_{s,3} > 0$ such that
	$$
		\epsilon_{H} \le C_{s,3}\bar{\sigma}\|\eta_k\|^{\omega},
	$$
	which gives $\|\eta_k\| \ge (\epsilon_{H} /(\bar{\sigma}C_{s,3}))^{1 /\omega}$. Let $C_s = (C_{s,2}^{-(2+\omega) /(1+\omega)} \w C_{s,3}^{-(2+\omega) /\omega}) \cdot \rh_1\ubar{\sigma} /(2+\omega)$. Then we get the result combining \cref{eq:s}.
\end{proof}

Plugging \cref{lem:sigmabar,lem:obj-dec-ar} into \cref{cor:K} gives the iteration complexity of RAR with $\ar{\omega}$ regularization.

\begin{theorem}[Iteration complexity of RAR] \label{thm:ar}
	Under \cref{asmp},
	\cref{alg:ar} finds an $(\epsilon_{g},\epsilon_{H})$-approximate second-order stationary point with the following worst-case iteration complexity:
	$$
		O\left(\max\left\{
		\epsilon_{g}^{ - (1+[\omega-\alpha]_{+})\cdot \tfrac{2+\omega}{1+\omega}} ,\
		\epsilon_{g}^{- \tfrac{(2+\omega)[\omega-\alpha]_+}{\omega}} \epsilon_{H}^{- \tfrac{2+\omega}{\omega}} ,\
		\epsilon_{H}^{ - \left(1+\tfrac{[\omega-\alpha]_{+}}{\alpha}\right) \tfrac{2+\omega}{\omega}}
		\right\}\right),
	$$
	where $\alpha = \mu\w\nu\w\theta$.
\end{theorem}
\begin{proof}
	By \cref{lem:sigmabar}, the second logarithmic term in \cref{cor:K} is suppressed by the first term.
	When $\epsilon_{H}\le \epsilon_{g}^{\alpha}$, \cref{lem:sigmabar} states that $\bar{\sigma} = C_{\sigma}\epsilon_{H}^{-[\omega-\alpha]_{+}/\alpha}$.
	We calculate the ratio of two minimal decreases for cases of $k\in \mathcal S_{GG}$ and $k\in \mathcal S_{L}$  in \cref{lem:obj-dec-ar}:
	$$
		\begin{aligned}
			\frac{(\epsilon_{g} / \bar{\sigma})^{(2+\omega) / (1+\omega)}}{(\epsilon_{H} / \bar{\sigma})^{(2+\omega) / \omega}}
			=   & C_{\sigma}^{(2+\omega)/(\omega(1+\omega))}  \frac{(\epsilon_{g}\epsilon_{H}^{- [\omega-\alpha]_{+} /\alpha} )^{(2+\omega) / (1+\omega)}}{(\epsilon_{H}^{1 - [\omega-\alpha]_{+}/\alpha})^{(2+\omega) / \omega}} \\
			\ge & \frac{(\epsilon_{H}^{1/\alpha- [\omega-\alpha]_{+} /\alpha} )^{(2+\omega) / (1+\omega)}}{(\epsilon_{H}^{1 - [\omega-\alpha]_{+}/\alpha})^{(2+\omega) / \omega}}                                                 \\
			=   & \left( \epsilon_{H}^{\omega-\alpha-[\omega-\alpha]_{+}-\omega\alpha} \right) ^{(2+\omega)/(\alpha\omega(1+\omega))},
		\end{aligned}
	$$
	where the inequality uses $C_{\sigma}\ge 1$ and $\epsilon_{H}\le \epsilon_{g}^{\alpha}$. Since $\omega-\alpha-[\omega-\alpha]_{+}-\omega\alpha< 0$ (which is easy to verify), the above ratio is greater than 1, making the minimal decrease for $k\in \mathcal{S}_{L}$ a lower bound for \cref{lem:obj-dec-ar}:
	$$
		f(x_k) - f(x_{k+1})
		\ge C_s  (\epsilon_{H} /\bar{\sigma})^{(2+\omega)/\omega}
		= C_s C_{\sigma}^{-\frac{2+\omega}{\omega}}
		\epsilon_{H}^{\left( 1 + \frac{[\omega-\alpha]_{+}}{\alpha} \right) \frac{2+\omega}{\omega} }.
	$$

	When $\epsilon_{H}> \epsilon_{g}^{\alpha}$, \cref{lem:sigmabar} states that $\bar{\sigma} = C_{\sigma}\epsilon_{g}^{-[\omega-\alpha]_{+}}$, and \cref{lem:obj-dec-ar} gives
	$$
		f(x_k) - f(x_{k+1}) \ge C_{s}
		\min \left\{ C_{\sigma}^{-\frac{2+\omega}{1+\omega}}\epsilon_{g}^{(1+[\omega-\alpha]_{+})\frac{2+\omega}{1+\omega}}, C_{\sigma}^{-\frac{2+\omega}{\omega}} \epsilon_{g}^{\frac{(2+\omega)[\omega-\alpha]_{+}}{\omega}}\epsilon_{H}^{ \frac{2+\omega}{\omega} }\right\}.
	$$
	Combining two cases gives the result.
\end{proof}

\begin{corollary}[Optimal iteration complexity of RAR]\label{cor:ar}
	When $\omega = \alpha$, \cref{thm:ar} achieves the optimal iteration complexity:
	$$
		O\left(\max \left\{ \epsilon_{g}^{- \tfrac{2+\alpha}{1+\alpha}}, \epsilon_{H}^{-\tfrac{2+\alpha}{\alpha}} \right\} \right).
	$$
\end{corollary}
\begin{proof}
	When $\alpha \ge \omega$, \cref{alg:ar} becomes
	$$
		O\left(\max \left\{ \epsilon_{g}^{- \tfrac{2+\omega}{1+\omega}}, \epsilon_{H}^{-\tfrac{2+\omega}{\omega}} \right\} \right),
	$$
	which monotonically decreases as $\omega$ increases.
	When $\alpha \le \omega$, \cref{alg:ar} becomes
	$$
		O\left(\!\max\left\{
		\epsilon_{g}^{ - \tfrac{(2+\omega)(1+\omega-\alpha)}{1+\omega}}\!,
		\epsilon_{g}^{- \tfrac{(2+\omega)(\omega-\alpha)}{\omega}} \epsilon_{H}^{- \tfrac{2+\omega}{\omega}}\!\!,
		\epsilon_{H}^{ - \frac{2+\omega}{\alpha}}
		\right\}\right)
		\ge
		O\left(\!\max\left\{
		\epsilon_{g}^{ - \tfrac{(2+\omega)(1+\omega-\alpha)}{1+\omega}}\!,
		\epsilon_{H}^{ - \frac{2+\omega}{\alpha}}
		\right\}\!\right),
	$$
	where the right hand side monotonically increases with $\omega$.
	Therefore, the complexity in \cref{thm:ar} is optimal when $\alpha = \omega$.
\end{proof}

\cref{cor:ar} agrees with \cref{lem:sigmabar} that $\omega > \alpha$ under-regularizes. It also suggests that $\omega < \alpha$ \textit{over-regularizes}, yielding a suboptimal complexity bound.

\begin{remark}
	\cref{cor:ar} requires knowledge of $\mu$ to achieve the optimal iteration complexity.
  Grapiglia and Nesterov~\cite{grapiglia2017RegularizedNewton} pioneered the development of \textit{universal} adaptive regularization methods that achieve the optimal iteration complexity without requiring knowledge of $\mu$.
	They consider a similar model problem to \cref{eq:model-ar} and show that for a $C^{2,\mu}$ objective, for a certain range of $\sigma$, solving the model with $\omega=1$ yields objective decrease comparable to that of solving the $\mu$-aware model with $\omega=\mu$ and $\sigma=C_{H}$.
	A line search procedure is then used to identify such a $\sigma$, with the total number of searches bounded by a logarithmic term.
	We believe that a universal RAR method can be developed by incorporating the techniques in \cite{grapiglia2017RegularizedNewton} into our framework.
\end{remark}%

\ifSubfilesClassLoaded{\bibliography{ic}}{}
\section{Riemannian Trust Region Methods} \label{sec:rtr}
Riemannian trust region (RTR) methods \citep{zhang2023RiemannianTrust,absilTrustRegionMethodsRiemannian2007} add a trust region constraint to the vanilla quadratic model function \cref{eq:model}.
Inspired by \cite{curtis2021TrustRegionNewtonCG}, we additionally introduce a small \textit{non-adaptive} quadratic regularization term to facilitate the iteration complexity analysis, transforming the Riemannian trust region model problem at $x_k\in\M$ into:
\begin{align}\label{eq:model-tr}
	\min_{\eta \in T_{x_k}\mathcal{M},~\|\eta\|\le\Delta_k} \quad \bar{m}^{\mathrm{tr}}_{x_k}(\eta;\Delta_k) =  f(x_k) + \left< \eta, \grad f(x_k) \right> + \frac{1}{2} \langle \eta,\hess f(x_k)[\eta]\rangle + \frac{1}{4}\epsilon_{H}\|\eta\|^2.
\end{align}
It is worth noting that we incorporate the trust region as a problem constraint, and thus the regularization term in the model function \cref{eq:model-arn} is $\phi^{\mathrm{tr}}(\eta) \coloneqq \epsilon_{H}\|\eta\|^2 /4$.
We inherit the method framework in \cref{alg:arn} {with some modifications to the termination criteria.}
When the solution to \cref{eq:model-tr} is situated at the trust region boundary, \cref{eq:tc-1} may not be satisfied. Therefore, we introduce a truncation termination criterion:
a {subproblem} solution is returned as soon as it reaches the trust region boundary, which gives
\begin{equation}\label{eq:tc-t}\tag{TC.T}
	\|\eta_k\| = \Delta_k.
\end{equation}
Otherwise, we utilize the residual termination criteria \cref{eq:tc-1} or \cref{eq:tc-2}.
In all cases, \cref{eq:tc-c} must also be satisfied.
We define two additional index sets:
$$
	\mathcal{B} \coloneqq \{ k\in \mathcal{K}: \|\eta_k\| = \Delta_k\}, \ \ \text{and}\ \
	\mathcal{I} \coloneqq \{ k\in \mathcal{K}: \|\eta_k\| < \Delta_k\}.
$$
{For RTR, we only require \cref{lem:g-bound} to hold when $\eta_k$ resides within the interior of the trust region, i.e., $\|\eta_k\|$ is strictly smaller than the trust region radius.
		Consistently, a Lanczos-based Krylov subspace method \citep{luksan2008Lagrangemultipliers,gould1999SolvingTrustRegion} satisfies the requirement (verified in \cref{sec:pf-g-bound}).}
A more practical truncated conjugate gradient (TCG) method \citep{steihaugConjugateGradientMethod1983,toint1981efficientsparsity} can also be used to solve the subproblem. We defer the detailed description of these subproblem solvers to the next section.
The only difference from a standard trust region subproblem solver is that here we pass a regularized Hessian $\bar{H}_k \coloneqq H_k + \frac{1}{2}\epsilon_{H}I$ to it.
The method is presented in \cref{alg:tr}.
When we set $\epsilon_H = 0$ and omit the second-order certification, our algorithm reduces to the classical one in \citet{absilTrustRegionMethodsRiemannian2007}.

\begin{algorithm}[!th]
	\caption{Riemannian Trust Region Method}
	\label{alg:tr}
	{
		\Input{Initial point $x_0\in\mathcal{M}$, tolorance $\epsilon_{g}, \epsilon_{H} \in (0,1)$, trust region adaptation parameters $\kappa_2 \ge 1 > \kappa_{1} > 0, \bar{\Delta} > 0, \Delta_0 \in (0, \bar{\Delta}]$, step acceptance parameter $\varrho\in [0,1/4)$}
		\Output{Last iterate $x_k$}
	}
	\For{$k=0,1,\dots$}{
		\uIf {$\|g_k\| > \epsilon_{g}$} { \label{line:tr_sub_beg}
			obtain $\eta_k$ by solving problem \eqref{eq:model-tr} with termination criteria~(TC.(\hyperref[eq:tc-1]{1}\tcor\hyperref[eq:tc-t]{T})\tcand\hyperref[eq:tc-c]{C})
		} \uElseIf {$H_k\not\succeq -\epsilon_{H}I$} {
			obtain $\eta_k$ by solving problem \eqref{eq:model-tr} with termination criteria~(TC.(\hyperref[eq:tc-2]{2}\tcor\hyperref[eq:tc-t]{T})\tcand\hyperref[eq:tc-c]{C})
		} \Else {
			\Return $x_k$\;
		} \label{line:tr_sub_end}
		compute $\rho_k$ using \eqref{eq:rho}\;
		set $\Delta_{k+1} = \left\{\begin{aligned}&\min \{ \bar{\Delta}, \kappa_2\Delta_k \}, &&\rho_k > 3 /4 \text{ and } \|\eta_k\|= \Delta_k, && \texttt{// very successful}\\ &\kappa_1 \Delta_k, && \rho_k < 1 /4, && \texttt{// unsuccessful}\\& \Delta_k, && \text{o.w.}&& \texttt{// successful}\end{aligned}\right.$\;
		set $x_{k+1} = \begin{cases}R_{x_k}(\eta_k), &\rho_k \ge \varrho,\\  x_k & \text{o.w.}\end{cases}$
	}
\end{algorithm}

\ifSubfilesClassLoaded{\bibliography{ic}}{}
\subsection{Iteration Complexity of RTR} 

As discussed in \cref{sec:ic-frame}, obtaining an iteration complexity bound of RTR only requires establishing a lower bound of $\Delta_k$ and the minimal decrease in $f$ for $k\in \mathcal{S}$.
Similar to the case of RAR (see \cref{sec:rar}), we need to carefully discuss the gradient dependency introduced by a general retraction in order to establish a lower bound of $\Delta_k$. Moreover, given that our problem does not possess $C^{2,1}$ smoothness, we employ certain strategies to derive the minimal successful decrease, which are essential for achieving the optimal complexity bound as shown in \citet{cartis2011OptimalNewtontype}.

\begin{lemma}[Trust region radius lower bound] \label{lem:deltamin}
	The trust region radius has a lower bound
	$$
		\ubar{\Delta} \coloneqq C_{\Delta}\epsilon_{H}^{\frac{1}{\scriptstyle\mu\w\nu}},
	$$
	where $C_\Delta$ is a positive constant.
\end{lemma}

\begin{proof}
	In this proof, we omit the superscript of $\bar{m}^{\mathrm{tr}}$ and the subscript $x_k$ of $\bar{m}_{x_k}$, $m_{x_k}$, and $R_{x_k}$.
	According to the update rule in \cref{alg:tr}, the trust region radius only shrinks when $\rho_k<1/4$. Therefore, we only need to establish the lower bound of $\Delta_k$ when $\rho_k < 1/4$.
	As we are considering the lower bound of the trust region radius, we only need to consider the case of  $\|\eta_k\|\le 1$.
	Similar to \cref{eq:u},
	when $\rho_k < 1 /4$, we have
	\begin{align}\label{eq:deltamin-1}
		\frac{3}{4}(\bar{m}(0) - \bar{m}(\eta_k)) + \frac{3}{16} \epsilon_{H}\|\eta_k\|^2
		\leqslant & C_{d}\left\|\eta_k\right\|^{2+\mu\w\nu} + C_{R}\|g_k\|\|\eta_k\|^{1+\nu},
	\end{align}
	where $C_{d}\coloneqq C_{H}+ (1+C_{R}\bar{\Delta}^{\nu})\beta_{H}C_{R}$.

	\case{\rom{1}} If $k\in \mathcal{I}$, \cref{eq:deltamin-1} and \cref{lem:g-bound} gives
	$$
		\frac{3}{16}\epsilon_{H} \le C_{d}\|\eta_k\|^{\mu\w\nu} + C_{R}(\beta_{H}+\epsilon_{H}/2)\|\eta_k\|^{\nu},
	$$
	which further gives
	\begin{equation}\label{eq:deltamin-2}
		\Delta_k \ge \|\eta_k\| \ge \left( \frac{3\epsilon_{H}}{16(C_{d} + C_{R}(\beta_H+\epsilon_{H} /2))} \right)^{\frac{1}{\mu\w\nu}}.
	\end{equation}

	\case{\rom{2}} If $k\in \mathcal{B}$, by \citep[Lemma 4.3]{nocedal2006Numericaloptimization}, the Cauchy condition \cref{eq:tc-c} and \eqref{eq:deltamin-1} further give
	\begin{align}\label{eq:deltamin-4}
		\frac{3}{8}\|g_k\| \cdot \min\left\{\Delta_k, \frac{\|g_k\|}{\beta_H+\epsilon_H/2} \right\} + \frac{3}{16}\epsilon_{H}\Delta_{k}^2 \le
		C_{d}\Delta_{k}^{2+\mu\w\nu} + C_{R}\|g_k\|\Delta_{k}^{1+\nu}.
	\end{align}
	\case{\rom{2}.\rom{1}} If $\Delta_k \ge \|g_k\|/(\beta_H+\epsilon_H/2)$, since $\Delta_k = \|\eta_k\| \le 1$, \cref{eq:deltamin-4} gives
	\begin{align}\notag
		\frac{3}{16}\epsilon_{H}\Delta_{k}^2 \le
		C_{d}\Delta_{k}^{2+\mu\w\nu} + C_{R}\|g_k\|\Delta_{k}^{1+\nu}
		\le (C_d + C_R(\beta_H+\epsilon_H/2))\Delta_k^{2+\mu\w\nu},
	\end{align}
	which further gives the result in \cref{eq:deltamin-2}.\\
	\case{\rom{2}.\rom{2}} If $\Delta_k < \|g_k\| /(\beta_H+\epsilon_H/2)$,
	\cref{eq:deltamin-4} can be reformulated as
	$$
		\|g_k\| \Delta_{k} \left(\frac{3}{8}-C_{R}\Delta_{k}^{\nu}\right) \le C_{d}\Delta_k^{2 + \nu\wedge \mu} - \frac{3}{16}\epsilon_{H}\Delta_k^2.
	$$
	\case{\rom{2}.\rom{2}.\rom{1}} If $\Delta_k\le (\frac{3}{8C_{R}})^{1 /\nu}$, we have
	\begin{equation*}
		C_{d}\|\Delta_k\|^{2 + \nu\wedge \mu} - \frac{3}{16}\epsilon_{H}\|\Delta_k\|^2 \ge 0
	\end{equation*}
	which gives $\Delta_k\ge (\frac{3\epsilon_{H}}{16C_d} )^{\frac{1}{\mu\w\nu}}$.
	\case{\rom{2}.\rom{2}.\rom{2}} The other case is just $\Delta_k> (\frac{3}{8C_{R}})^{1 /\nu}$.
	Therefore, for \case{\rom{2}.\rom{2}}, we have
	\begin{equation}\label{eq:deltamin-3}
		\Delta_k \ge \min \left\{\left(\frac{3}{8C_{R}}\right)^{\frac{1}{\nu}}, \left( \frac{3\epsilon_{H}}{16C_{d}} \right)^{\frac{1}{\mu\w\nu}}\right\}.
	\end{equation}
	Combining \cref{eq:deltamin-2,eq:deltamin-3},
	and the initial assumption $\|\eta_k\|\le 1$ gives
	$$
		\begin{aligned}
			\Delta_k \ge & \min \left\{1, \left(\frac{3}{8C_{R}}\right)^{\frac{1}{\nu}}, \left( \frac{3\epsilon_{H}}{16C_{d}} \right)^{\frac{1}{\mu\w\nu}},
			\left( \frac{3\epsilon_{H}}{16(C_{d} + C_{R}(\beta_H+\epsilon_{H}/2))} \right)^{\frac{1}{\mu\w\nu}}
			\right\}.
		\end{aligned}
	$$
	Combining the above cases, we complete the proof with $C_{\Delta}$ defined accordingly.
\end{proof}

\begin{lemma}[Minimal successful decrease of RTR]\label{lem:obj-dec-tr}
	There exists a positive constant $C_g$ such that for any $k\in \mathcal{S}$, we have
	$$
		f(x_k) - f(x_{k+1}) \ge \frac{\varrho\epsilon_{H}}{4}\begin{cases}
			\displaystyle \Delta_{k}^{2} , \quad                                                                                  & k\in \mathcal{B} \cap \mathcal{S},     \\
			C_{g}^2\displaystyle \min \{ \|g_{k+1}\|^{4 /(1+\alpha)}\epsilon_{H}^{-2 /\alpha}, \epsilon_{H}^{2 /\alpha} \}, \quad & k\in \mathcal{I}\cap \mathcal{S}_{G},  \\
			\displaystyle (\epsilon_{H} /2)^{2 /\theta} , \quad                                                                   & k\in \mathcal{I} \cap \mathcal{S}_{L}.
		\end{cases}
	$$
\end{lemma}
\begin{proof}
	First of all, since the Cauchy termination criterion \cref{eq:tc-c} is enforced for all cases, for any $k\in \mathcal{S}$, we have
	\begin{equation}\label{eq:s-dec-0}
		f(x_k) - f(x_{k+1}) \ge \varrho(m_{x_k}(0) - m_{x_k}(\eta_k))
		= \varrho(\bar{m}_{x_k}(0)-\bar{m}_{x_k}(\eta_k) + \epsilon_{H}\|\eta_k\|^2/4)
		\ge \frac{\varrho\epsilon_{H}}{4}\|\eta_k\|^2.
	\end{equation}
	Therefore, we only need to figure out a lower bound of $\|\eta_k\|$.

	\case{\rom{1}} For $k\in \mathcal{B}\cap \mathcal{S}$, we have $\|\eta_k\| = \Delta_k$.

	\case{\rom{2}} For $k\in \mathcal{I}\cap \mathcal{S}_{G}$, residual termination condition \cref{eq:tc-1} is used.\\
	\case{\rom{2}.\rom{1}} If $\epsilon_{H}\|\eta_k\| \le \epsilon_{H}^{\frac{1+\alpha}{2\alpha}}\|\eta_k\|^{\frac{1+\alpha}{2}}$, we have $\|\eta_k\|\le 1$, which together with \cref{lem:tc,lem:g-bound} gives
	\begin{equation} \label{eq:s-dec-5}
		C_{g}'\|\eta_k\|^{1+\alpha} + \frac{1}{2}\epsilon_{H}^{\tfrac{1+\alpha}{2\alpha}}\|\eta_k\|^{\tfrac{1+\alpha}{2}} - \|g_{k+1}\| \ge 0,
	\end{equation}
	where $C_{g}'\coloneqq 2C_{H}(1+C_{R})+\beta_H C_{R}+(\beta_H+\epsilon_{H}/2)^{1+\theta}$.
	By \citep[Lemma 17]{royer2018Complexityanalysis}, the solution of this inequality indicates
	\begin{align}\notag
		\|\eta_k\|
		 & \ge \left( \frac{-1 + \sqrt{1 + 16C_{g}'\|g_{k+1}\|\epsilon_{H}^{-\frac{1+\alpha}{\alpha}}}}{4C_g'}\epsilon_{H}^{\frac{1+\alpha}{2\alpha}} \right)^{\tfrac{2}{1+\alpha}}                            \\\notag
		 & \ge \left( \frac{-1 + \sqrt{1 + 16C'_{g}}}{4C'_g}\min \left\{ \|g_{k+1}\|\epsilon_{H}^{-\tfrac{1+\alpha}{2\alpha}}, \epsilon_{H}^{\tfrac{1+\alpha}{2\alpha}} \right\} \right)^{\tfrac{2}{1+\alpha}} \\\label{eq:s-dec-4}
		 & \eqqcolon C_{g} \min \{ \|g_{k+1}\|^{2 /(1+\alpha)}\epsilon_{H}^{-1 /\alpha}, \epsilon_{H}^{1 /\alpha} \}.
	\end{align}
	\case{\rom{2}.\rom{2}} If $\epsilon_{H}\|\eta_k\| > \epsilon_{H}^{\frac{1+\alpha}{2\alpha}}\|\eta\|^{\frac{1+\alpha}{2}}$ , we have
	$$
		\|\eta_k\|^{1 - \tfrac{1+\alpha}{2}} > \epsilon_{H}^{\tfrac{1+\alpha}{2\alpha} - 1}
		\quad\Longrightarrow\quad
		\|\eta_k\| > \epsilon_{H}^{1 /\alpha}.
	$$
	Set $C_g \coloneqq  C_{g} \w 1$. Then
	\case{\rom{2}.\rom{2}} also satisfies \cref{eq:s-dec-4}.

	\case{\rom{3}} For $k\in \mathcal{I}\cap \mathcal{S}_{L}$, since the algorithm does not terminate, residual termination condition \cref{eq:tc-2} is used and \cref{lem:tc} gives
	$$
		\epsilon_{H} < -\lambda_{\min}(H_k) \le \|\eta_k\|^{\theta} + \frac{\epsilon_{H}}{2},
	$$
	which further gives $\|\eta_k\| \ge (\epsilon_{H} /2)^{1 /\theta}$.
	Combining \cref{eq:s-dec-0} and \case{\rom{1}-\rom{3}} gives the result.
\end{proof}
We should note that the termination criterion \cref{eq:tc-2} might appear less practical in the context of RTR because of the fixed nature of $\hess \bar{m}^{\mathrm{tr}}$. However, when this criterion is applied, i.e., when $H_k \not\succeq -\epsilon_{H}I$, we anticipate that the Krylov subspace method will be capable of identifying the negative curvature of the model problem and return a solution located on the trust region boundary.
Therefore, we can omit \cref{eq:tc-2} and rely solely on \cref{eq:tc-t}.
To be consistent with RAR, here we retain \cref{eq:tc-2} as an early stopping criterion, and its inexactness can potentially expedite the process. We will give a concrete subproblem solver and an alternative termination criterion to replace \cref{eq:tc-2} for RTR in \cref{sec:tr-oc}.

Plugging \cref{lem:deltamin,lem:obj-dec-tr} into \cref{cor:K} gives the iteration complexity of RTR.

\begin{theorem}[Iteration complexity of RTR] \label{thm:tr}
	Under \cref{asmp},
	\cref{alg:tr} finds an $(\epsilon_{g},\epsilon_{H})$-approximate second-order stationary point with the following worst-case iteration complexity:
	$$
		O\left(\max\left\{
		\epsilon_{H}^{-1-\tfrac{2}{\alpha}},
		\epsilon_{g}^{-\tfrac{4}{1+\alpha}}\epsilon_{H}^{-1 + \tfrac{2}{\alpha}} \right\}\right),
	$$
	where $\alpha = \mu\w\nu\w\theta$.
\end{theorem}
\begin{proof}
	By \cref{lem:deltamin}, the second logarithmic term in \cref{cor:K} is suppressed by the first term.
	For $k\in \mathcal{S}_{GG}$, by \cref{lem:deltamin,lem:obj-dec-tr}, we have
	$$
		f(x_k) - f(x_{k+1}) \!\ge \frac{\varrho}{4}\min \left\{\epsilon_{H}\ubar{\Delta}^2, C_{g}^2 \epsilon_{g}^{4 / (1+\alpha)}\epsilon_{H}^{1-2 /\alpha}, C_{g}^2\epsilon_{H}^{1+2 /\alpha}  \right\}
		\!\gtrsim \min \left\{\epsilon_{H}^{1+2 /\alpha}, \epsilon_{g}^{4 / (1+\alpha)}\epsilon_{H}^{1-2 /\alpha} \right\}.
	$$
	For $k\in \mathcal{S}_{L}$, by \cref{lem:deltamin,lem:obj-dec-tr}, we have
	$$
		f(x_k) - f(x_{k+1}) \!\ge \frac{\varrho}{4} \min \left\{\epsilon_{H}\ubar{\Delta}^2, 2^{-2 /\theta}\epsilon_{H}^{1 + 2/\theta}\right\} \gtrsim \epsilon_{H}^{1+2 /\alpha}.
	$$
	Combining two cases with \cref{cor:K} gives the result.
\end{proof}

\begin{corollary}[Optimal iteration complexity of RTR]\label{cor:tr}
	By choosing $\epsilon_{H} = \epsilon_{g}^{\alpha /(1+\alpha)}$, \cref{thm:tr} achieves the optimal (first-order) iteration complexity:
	\begin{equation*}
		O\left(\epsilon_{g}^{-\tfrac{{2 + \alpha}}{1+\alpha}}\right).
	\end{equation*}
\end{corollary}

Our analysis, especially in light of \cref{cor:ar,cor:tr}, reveals that the parameters $\mu$, $\nu$, and $\theta$ are \textit{equivalently} responsible for governing the iteration complexity of RARN, and this control is encapsulated by the single parameter $\alpha = \mu\w\nu\w\theta$. In essence, elevating any of these parameters in isolation will not improve the worst-case iteration complexity.
Furthermore, this observation can serve as a valuable guideline for algorithm design: when $\mu$, a value determined by the problem, is relatively large, we need to increase $\nu$ and $\theta$ to match the accuracy of the model function, which approximates the objective function, to achieve the optimal iteration complexity. Conversely, a smaller $\mu$ provides the flexibility to relax the smoothness requirement on the retraction and the precision requirement on the subproblem solver, thus effectively curtailing computational costs.

\begin{remark}
	Our \cref{alg:tr} shares a structural similarity with the method in \cite{curtis2021TrustRegionNewtonCG}, as both follow the standard trust region method procedure, and incorporate a non-adaptive quadratic regularization term into the quadratic model \cref{eq:model-tr}. However, key differences arise in other algorithmic components and analysis. First, our termination criteria \cref{eq:tc-1,eq:tc-2} allow for more controlled inexactness through parameters $\theta_1$ and $\theta_2$, which appear explicitly in the complexity bounds. 
	These criteria are carefully designed for $C^{2,\mu}$ objectives, with the optimal inexactness levels $\theta_1$ and $\theta_2$ depending explicitly on $\mu$.
	Second, our analysis accommodates general retractions $R$ with a \holder-continuous differential, whose smoothness order $\nu$ is likewise treated as an approximation parameter, with the optimal level $\nu$ depending explicitly on $\mu$.
	This treatment and analysis of general retractions marks a significant departure from the Euclidean approach in \cite{curtis2021TrustRegionNewtonCG}.
\end{remark}%

\ifSubfilesClassLoaded{\bibliography{ic}}{}
\section{Subproblem Solvers and Operation Complexity} 

In this section, we comprehensively analyze the subproblem solvers employed in RAR and RTR. Our focus centers on introducing Lanczos-based Krylov subspace methods as the central subproblem solver, complemented by integrating minimal eigenvalue oracles (MEO) to assess second-order stationarity.
Building upon these subproblem solvers, we provide insights into our algorithms' operation complexity, quantified in terms of the number of Hessian-vector products.

\subsection{Lanczos-Based Krylov Subspace Methods}\label{sec:kry}

We begin by restating the RARN subproblem:
\begin{equation}\label{eq:model-arn-2}
	\min_{\eta\in \tm{x}} \bar{m}(\eta)\coloneqq \left<\eta,g \right> + \frac{1}{2}\left<\eta,H\eta \right> + \phi(\|\eta\|,\sigma).
\end{equation}
In this formulation, we omit the reference to a specific point $x\in\M$ and the constant term $f(x)$. Additionally, we simplify the regularization function to only depend on the norm of $\eta$.
Krylov subspace methods aim to approximately minimize $\bar{m}$ by identifying solutions within specific subspaces known as the Krylov subspaces.
In this paper, for the automatic fulfillment of the Cauchy condition \cref{eq:tc-c}, we construct the Krylov subspaces based on $(H,g)$. Specifically, the order-$j$ Krylov subspace is defined as $\spa\{ g,Hg, \ldots ,H^{j-1}g \}$.
We denote $\xi_j$ as the order-$j$ Krylov subspace solution to \cref{eq:model-arn-2}.
Consequently, the Cauchy point $\eta^{\mathrm{C}} = \xi_1$ is the first-order Krylov subspace solution.
To efficiently compute the Krylov subspace solutions, we use the Lanczos method to construct an orthogonal Krylov subspace basis $Q_{j} = (q_0, \ldots, q_{j-1})$, transforming the $j$th subproblem iteration on the tangent space into the following one in $\R^{j}$:
$$
	\min_{u \in\R^{j}} \|g\|\cdot (u)_{1} + \frac{1}{2}u^*Q_{j}^* H Q_{j} u + \phi(\|u\|,\sigma),
$$
where $(u)_1 \in \R$ is the first component of $u\in \R^{j}$, $u^*$ is the transpose of $u$, and $Q^{*}_j:\tm{x}\to\R^{j}$ is the {adjoint} of $Q_{j}$, and we use the fact that $Q^*_{j}g = \|g\|e_{1}$ and $\|Q_{j}u\| = \|u\|$.
The tridiagonal structure of $Q^*_{j}HQ_{j}$ streamlines the efficient solution of the above problem, requiring only $O(1)$ Hessian-vector products.
For a more detailed explanation of Lanczos-based Krylov subspace methods, please refer to, e.g., {\citet[Section 8]{agarwal2021Adaptiveregularization}}, \citet[Chapter 10]{golub2013Matrixcomputationsa}, and \citet[Lecture 36]{trefethen2022Numericallinear}.
We present several conditioned operation complexities of Krylov subspace methods.

\begin{proposition}[Operation complexity of Krylov subspace methods]\label{prop:krylov}
	Suppose $\eta_k$ is returned by a Krylov subspace method, and $\eta^{*}$ is the exact solution to \cref{eq:model-arn-2}. We have
	\begin{enumerate}
		\item {with a perturbed Krylov subspace basis%
		      \footnote{For RAR, we directly invoke a Krylov subspace method, which may encounter ``hard cases'' when $\left< g_k,\nu_k \right>$ with a regular basis discussed in \cref{sec:kry}, where $\nu_{k}$ is the corresponding eigenvector of $\lambda_{\min}(H_k)$.
			      Therefore, we use a perturbed Krylov subspace basis for RAR to escape the hard case and thus provide an explicit maximal number of subproblem iterations. Please refer to \cref{sec:prop-krylov} and \cite[Section 5.2]{carmon2020FirstOrderMethods} for more details.
		      }%
		      ,} for the adaptive regularization subproblem, $\eta_k$ is an $\epsilon$-optimal solution to \cref{eq:model-arn-2} with {probability at least $1-\delta$ and} at most ${C_{\mathrm{sub},1}\|\eta^*\|\epsilon^{-1 /2}\log(n /\delta)}$ Hessian-vector product operations, {where $n$ is the manifold dimension} \citep[Corollary 4.2]{carmon2020FirstOrderMethods}%
		      \footnote{\citet{carmon2020FirstOrderMethods} considers cubic regularization with fixed regularization parameter, but it still applies here because the implied complexity does not involve the regularization order or parameter.
		      }%
		      ; \label{prop:krylov-ar}
		\item for the adaptive regularization subproblem, if ${H}\succeq \epsilon_{H}I$, $\eta_k$ is an $\epsilon$-optimal solution to \cref{eq:model-arn-2} with at most $C_{\mathrm{sub},2}\lambda_{\min}(H)^{-1 /2} \log \epsilon^{-1}$ Hessian-vector product operations \citep[Corollary 4.2]{carmon2020FirstOrderMethods};\label{prop:krylov-ar-2}
		\item for the \textup{non-regularized} trust region subproblem, if $H\succeq \epsilon_{H}I$, $\eta_k$ is an $\epsilon$-optimal solution to \cref{eq:model-arn-2} with at most $C_{\mathrm{sub},2}\lambda_{\min}({H})^{-1 /2} \log \epsilon^{-1}$ Hessian-vector product operations \citep[Theorem 4.1]{carmon2020FirstOrderMethods}.\label{prop:krylov-tr}
	\end{enumerate}
	{
	Constants $C_{\mathrm{sub},1}$ and $C_{\mathrm{sub},2}$ are given by
	\begin{align*}
		C_{\mathrm{sub},1} & =
		2\sqrt{\beta_{H}\left( 2 + \log^{2}(2\sqrt{n}/\delta) \right)}, \\
		C_{\mathrm{sub},2} & =
		\frac{\sqrt{\beta_{H}}}{4} \left(1 + \log \left( 36 \left(1 + \frac{\beta_{H}}{2} + \frac{\bar{\sigma}}{1+\alpha} + \bar{\Delta} + \frac{\beta_{H}+2\epsilon_{H}}{2}\bar{\Delta}^{2}\right)\left(1 \vee \beta_{g} \vee \left(\frac{3(\beta_{H}+1)}{\ubar{\sigma}}\right)^{1 /\alpha} \right)^{2+\alpha} \right)\right)
		.\end{align*}
	}
\end{proposition}

{
Some further discussion on this proposition is provided in \cref{sec:prop-krylov} and \citet{carmon2020FirstOrderMethods}.
}

\subsection{Minimal Eigenvalue Oracle}\label{sec:meo}

We have not yet specified how we test for second-order stationarity and what to do when $H \not\succeq \epsilon_{H}I$ in \cref{prop:krylov-ar-2,prop:krylov-tr} of \cref{prop:krylov}.
To address this, we assume our algorithm can access a minimal eigenvalue oracle (MEO). Following \citet{curtis2021TrustRegionNewtonCG}, we require the MEO to indicate whether $H \succeq -\epsilon_{H}I$ or, if not, to return a unit vector $\eta^{\mathrm{E}}$ such that
\begin{equation}\label{eq:tc-e}\tag{TC.E}
	\left<\eta^{\mathrm{E}}, g \right> \le 0 \quad\text{and}\quad \left<\eta^{\mathrm{E}},H\eta^{\mathrm{E}} \right> \le -\frac{1}{2}\epsilon_{H}\|\eta^{\mathrm{E}}\|^2,
\end{equation}
with at most $O(n\w \epsilon_{H}^{-1 /2}\log(1 /\delta))$ Hessian-vector products, where $n$ is the manifold dimension and $\delta$ is the probability that the MEO incorrectly claims $H \succeq -\epsilon_{H}$.
Actually, Krylov subspace methods, with a Lanczos-based orthogonal basis or an $H$-conjugate basis (resulting in conjugate gradient methods), satisfy these requirements with an operation complexity of $O(n\w \epsilon_{H}^{-1 /2}\log(n / \delta^2))$ \citep[Appendix B.2]{royer2020NewtonCGAlgorithm}.

\subsection{Operation Complexity of RAR}
In this subsection, we provide a concrete algorithm of adaptive regularization methods with a strong operation complexity guarantee.
To simplify the analysis, we set $\alpha \coloneqq \mu = \nu = \theta = \omega$, the optimal parameter profile discussed in previous sections.
As for the subproblem solver, we substitute Lines \ref{line:ar_sub_beg} to \ref{line:ar_sub_end} of \cref{alg:ar} with \cref{alg:ar-sub}.
Based on \cref{prop:krylov-ar} of \cref{prop:krylov}, we limit the number of subproblem iterations to $K_{\mathrm{sub}} = {O}(\epsilon_{H}^{-1 /2})$, and propose a new termination criterion:
\begin{equation}\label{eq:tc-d}\tag{TC.D}
	m(0) - m(\eta) \ge \frac{\alpha\epsilon_{H}^{(2+\alpha) / \alpha}}{12\bar{\sigma}^{2 /\alpha}}.
\end{equation}
which directly fulfills the requirement for minimal successful decrease (\cref{lem:obj-dec-ar}, albeit potentially with a different constant), rendering \cref{eq:tc-2} obsolete.
We have the following proposition.

\begin{proposition}\label{prop:tcd}
	If $H_k \not\succeq -\epsilon_{H}I$, then a Krylov subspace method {with a perturbed basis} will find a solution satisfying \cref{eq:tc-d} {with probability at least $1-\delta$ and} within at most ${C_{\mathrm{sub}}^{\mathrm{ar}} }\epsilon_{H}^{-1 /2}$ iterations, {where
			\[
				C_{\mathrm{sub}}^{\mathrm{ar}} =
				\sqrt{\frac{48\beta_{H}\bar{\sigma}}{\alpha \ubar{\sigma}}\left( 2 + \log^{2}\frac{2\sqrt{n}}{\delta} \right)} + \frac{1}{2}
				,\]
			where $n$ is the dimension of the manifold.
		}
\end{proposition}

{
\cref{prop:tcd} is proved
in \cref{sec:prop-tcd}.%
}
{We now break down \cref{alg:ar-sub} line by line.}
We first try solving the subproblem using a Lanczos-based Krylov subspace method (Lines \ref{line:ar-lk-beg} to \ref{line:ar-lk-end}) {with a perturbed subspace basis}.
By \cref{prop:tcd}, if \texttt{max\_flag} remains \texttt{true} after $K_{\mathrm{sub}}{= \lceil C_{\mathrm{sub}}^{\mathrm{ar}} \epsilon_{H}^{-1 /2} \rceil}$ iterations, we know that $H_k\succeq -\epsilon_{H}I$ {(Lines \ref{line:ar-max-beg} to \ref{line:ar-max-end})}.
Subsequently, if $\|g_k\| \le \epsilon_{g}$, it signifies that $x_k$ is an $(\epsilon_{g},\epsilon_{H})$-approximate second-order stationary point and we can terminate the algorithm {(\cref{line:ar-return-term})}.
If $\|g_k\| > \epsilon_{g}$, {because the Krylov subspace method fails to meet an appropriate termination criterion within the specified number of iterations,} we still need to efficiently find an iteration step that aims for the first-order stationarity.
Since now $H_k\succeq -\epsilon_{H}I$, we solve \cref{eq:model-arn-2} again with a regularized Hessian $\bar{H}_k = H_k + 2\epsilon_{H}I$ to obtain $\eta_k$ (\cref{line:ar-return-recal}),
which is guaranteed to satisfy the residual termination criterion \cref{eq:tc-1} with at most $\widetilde{O}(\epsilon_{H}^{-1 /2})$ iterations by \cref{prop:krylov-ar-2} of \cref{prop:krylov}.

On the other hand, even if \texttt{max\_flag = false}, we cannot assert that $H_k \not\succeq -\epsilon_{H}I$ (Lines \ref{line:ar-res-beg} to \ref{line:ar-res-end}). Therefore, when \texttt{max\_flag = false} and the first-order stationarity is achieved ($\|g_k\|\le \epsilon_g$), we still need to call an MEO to test the second-order stationarity. If the MEO claims the presence of the second-order stationarity, we terminate the algorithm  (\cref{line:ar-return-meo}); otherwise, we continue the algorithm with the returned subproblem solution $\eta_k=\xi_{j}$ (\cref{line:ar-return-cont}).

\begin{algorithm}[!th]
	\SetKw{Break}{break for}
	\caption{Subproblem solver for RAR} \label{alg:ar-sub}
	{
		\Input{Tangent space $\tm{x}$, $g_k$, $H_k$, tolorance $\epsilon_{g},\epsilon_{H}\in (0,1)$, interation constant $C_{\mathrm{sub}}^{\mathrm{ar}}$, MEO}
		\Output{$\eta_k$ or $x_k$}
	}
	set \texttt{max\_flag = true}\;
	\For {$j = 1, ..., K_{\mathrm{sub}}{=\lceil C_{\mathrm{sub}}^{\mathrm{ar}} \epsilon_{H}^{-1/2}\rceil}$} { \label{line:ar-lk-beg}
		get order-$j$ Krylov subspace solution $\xi _{j}$\;
		\If {($\|g_k\|> \epsilon_g$ and $\xi_j$ satisfies \cref{eq:tc-1}) or $\xi_{j}$ satisfies \cref{eq:tc-d}} { \label{line:tcd}
			set \texttt{max\_flag = false}\;
			\Break
		}
	} \label{line:ar-lk-end}
	\uIf {\texttt{max\_flag = true}} { \label{line:ar-max-beg}
		\uIf {$\|g_k\| \le \epsilon_{g}$} {
			\Return $x_k$ and terminate the outer algorithm \label{line:ar-return-term}
		} \Else {
			\Return $\eta_k$ by solving problem \cref{eq:model-ar} with $H_k$ replace with $\bar{H}_k = H_k + 2\epsilon_{H}I$ and termination criterion \cref{eq:tc-1} \label{line:ar-return-recal}
		} \label{line:ar-max-end}
	} \Else (\tcp*[h]{\texttt{max\_flag = false}}) {  \label{line:ar-res-beg}
		\uIf {$\|g_k\| \le \epsilon_{g}$} {
			call an MEO to test the second-order stationarity\;
			\If {the MEO indicates that $H_k \succeq - \epsilon_{H}I$} {
				\Return $x_k$ and terminate the outer algorithm \label{line:ar-return-meo}
			}
		} \Else {
			\Return $\eta_k = \xi_j$ \label{line:ar-return-cont}
		}
	} \label{line:ar-res-end}
\end{algorithm}

Due to the utilization of a regularized Hessian (\cref{line:ar-return-recal}), a minor adjustment is required for the iteration complexity of RAR (\cref{cor:ar}).
When combined with the operation complexity of Lanczos-based Krylov subspace methods and MEO, we obtain an operation complexity guarantee of RAR.
\begin{corollary}[Operation complexity of RAR] \label{cor:ar-oc}
	\cref{alg:ar,alg:ar-sub} finds an $(\epsilon_{g},\epsilon_{H})$-approximate second-order stationary point with the following worst-case operation complexity:
	$$
		\widetilde{O}\left(\max \left\{ \epsilon_{g}^{- \tfrac{2(2+\alpha)}{1+\alpha}}\epsilon_{H}^{\tfrac{2+\alpha}{\alpha}}, \epsilon_{H}^{-\tfrac{2+\alpha}{\alpha}} \right\} \cdot \epsilon_{H}^{-1 /2} \right),
	$$
	where $\widetilde{O}$ suppresses the logarithmic dependency on $\epsilon_{H}$.
	When $\epsilon_{H} = \epsilon_{g}^{\alpha /(1+\alpha)}$, the complexity becomes
	$$
		\widetilde{O}\left(\epsilon_{g}^{- \tfrac{4+3\alpha}{2(1+\alpha)}} \right),
	$$
	which further becomes $\widetilde{O}(\epsilon_{g}^{-7/4})$ when $\alpha = 1$.
\end{corollary}

We thoroughly examine in \cref{sec:valid-aroc} the applicability of \cref{lem:sigmabar,lem:obj-dec-ar}, and thus \cref{cor:ar}, when employing \cref{alg:ar-sub}, making this operation complexity valid.

\subsection{Operation Complexity of RTR} \label{sec:tr-oc}

Similarly, we adopt the parameter profile $\alpha \coloneqq \mu = \nu = \theta$ in this subsection.
For RTR, we substitute Lines \ref{line:tr_sub_beg} to \ref{line:tr_sub_end} of \cref{alg:tr} with \cref{alg:tr-sub}.
Additionally, to ensure $\bar{H}_k \succeq \epsilon_{H}I$ indicates $H_k \succeq -\epsilon_{H}I$, we increase the regularization strength by setting the regularized Hessian as $\bar{H}_k = H_k + 2\epsilon_{H}I$, which aligns with its usage in \cref{alg:ar-sub}.
This modification deprecates termination criterion \cref{eq:tc-2}.
Based on \cref{prop:krylov-tr} of \cref{prop:krylov} applied to the new regularized Hessian $\bar{H}_k$, we limit the number of subproblem iterations to $K_{\mathrm{sub}} = \widetilde{O}(\epsilon_{H}^{-1 /2})$, which is sufficient to meet {the residual termination criterion \cref{eq:tc-1} if $\bar{H}_k \succeq \epsilon_{H}I$ and $g_k\ne 0$.
	}

	{
		\begin{proposition}\label{prop:tr-sub}
			If $H_k \succeq -\epsilon_{H}I$ and $\|g_k\| > \epsilon_{g}$, then a Krylov subspace method will find a solution satisfying \cref{eq:tc-1} within at most ${C_{\mathrm{sub}}^{\mathrm{tr}} }\epsilon_{H}^{-1 /2}$ iterations, where
			\begin{align*}
				C_{\mathrm{sub}}^{\mathrm{tr}} = &
				\frac{\sqrt{\beta_{H}}}{4} \left(1 + \log \left( 36 \left(1 + \frac{\beta_{H}}{2} + \frac{\bar{\sigma}}{1+\alpha} + \bar{\Delta} + \frac{\beta_{H}+2\epsilon_{H}}{2}\bar{\Delta}^{2}\right)\left(1 \vee \beta_{g} \vee \left(\frac{3(\beta_{H}+1)}{\ubar{\sigma}}\right)^{1 /\alpha} \right)^{2+\alpha} \right)\right) \\
				                                 & \cdot \log\left(\frac{\epsilon_{H}\epsilon_{g}^{2+2\theta_1} }{2(\beta_{H}+2\epsilon_{H})^{4+2\theta_1}}\right)^{-1}
				.\end{align*}
		\end{proposition}
	}

	{
		\cref{prop:tr-sub} is proved
		in \cref{sec:prop-tr-sub}.
		However, when $g_k = 0$, failure to meet
		\cref{eq:tc-1} within $K_{\mathrm{sub}}$ iterations does not necessarily imply that $H_k \not\succeq -\epsilon_{H}I$.
		Similarly, when $\|g_k\|=0$, we cannot claim that $H_k \succeq -\epsilon_{H}I$ when \cref{eq:tc-1} is met within $K_{\mathrm{sub}}$ iterations.
		Therefore, we directly call an MEO to test the second-order stationarity when $\|g_k\|\le \epsilon_{g}$ (Lines~\ref{line:tr-g-beg} to \ref{line:tr-g-end}).
		When making use of the return of an MEO, we set $\eta_k = \Delta_{k}\eta^{\mathrm{E}}$ to satisfy \cref{eq:tc-t} and thus replace \cref{eq:tc-2}.
	}

	{If $\|g_k\| > \epsilon_{g}$,}
we first try solving the subproblem using a Lanczos-based Krylov subspace method {with $K_{\mathrm{sub}} = \lceil C_{\mathrm{sub}}^{\mathrm{tr}} \epsilon_{H}^{-1 /2}\rceil$} (Lines \ref{line:tr-lk-beg} to \ref{line:tr-lk-end}).
	{Note that here we can use a regular Krylov subspace basis instead of a perturbed one, as \cref{prop:tr-sub} applies to all cases.}
	{If \texttt{max\_flag = false}, we proceed with the returned subproblem solution $\eta_k = \xi_j$ (\cref{line:tr-return-cont}).}
If \texttt{max\_flag} remains \texttt{true} after $K_{\mathrm{sub}}$ iterations, we know that $H_k\not\succeq -\epsilon_{H}I$, and we use an MEO to find {a suitable iteration step corresponding to} its minimal eigenvalue (Lines \ref{line:tr-max-beg} to \ref{line:tr-max-end}).

\begin{algorithm}[!th]
	\SetKw{Break}{break for}
	\caption{Subproblem solver for RTR} \label{alg:tr-sub}

	{
		\Input{Tangent space $\tm{x}$, $g_k$, $H_k$, $\Delta_k$, tolorance $\epsilon_{g},\epsilon_{H}\in (0,1)$, interation constant $C_{\mathrm{sub}}^{\mathrm{tr}}$, MEO}
		\Output{$\eta_k$ or $x_k$}
	}
	set \texttt{max\_flag = true}\;
	\If {$\|g_k\|\le \epsilon_{g}$} { \label{line:tr-g-beg}
		call an MEO to test the second-order stationarity\;
		\uIf {the MEO indicates that $H_k \succeq - \epsilon_{H}I$} {
			\Return $x_k$ and terminate the outer algorithm
		} \Else {
			\Return $\eta_k = \Delta_k \eta^{\mathrm{E}}$, where $\eta^{\mathrm{E}}$ is obtained through the MEO \label{line:tr-g-end}

		}
	}
	\For {$j = 1, ..., K_{\mathrm{sub}}{=\lceil C_{\mathrm{sub}}^{\mathrm{tr}}\epsilon_{H}^{-1 /2}\rceil}$} { \label{line:tr-lk-beg}
		get order-$j$ Krylov subspace solution $\xi_{j}$ with the regularized Hessian $\bar{H}_k$\;
		\If {$\xi_{j}$ satisfies (\ref{eq:tc-1}\tcor\hyperref[eq:tc-t]{T})} {
			set \texttt{max\_flag = false}\;
			\Return $\eta_k = \xi_j$ \label{line:tr-return-cont} \label{line:tr-res-end}
		}
	} \label{line:tr-lk-end}
	\tcp*[h]{\texttt{max\_flag = true}} \;
	call an MEO to return $\eta^{\mathrm{E}}$\label{line:tr-max-beg}\;
	\Return $\eta_k = \Delta_k \eta^{\mathrm{E}}$ \label{line:tr-max-end}
\end{algorithm}

We remark that in the subproblem solver for RAR (\cref{alg:ar-sub}), we exclusively employ an MEO to assess second-order stationarity and do not rely on it to provide an iteration step. This is because for the RAR subproblem, a Krylov subspace method performs well (efficiently finds a solution that offers sufficient model decrease) when $H_k \not\succeq -\epsilon_{H}I$ (see \cref{prop:tcd}).
However, Krylov subspace methods for the RTR subproblem do not enjoy this nice property when $H_k \not\succeq -\epsilon_{H}I$, especially in \textit{hard cases} (see \cite[Chapter 4]{nocedal2006Numericaloptimization} or \cite[Chapter 7]{conn2000Trustregion}).
Therefore, we also incorporate an MEO into our subproblem solver for RTR to determine the iteration steps.
This approach is characterized by its simplicity and efficiency (recall that a Krylov subspace method is an MEO), aligning with established practices in the field \citep{curtis2021TrustRegionNewtonCG,xu2020Newtontypemethods}.

By \cref{cor:tr} and the operation complexity of Lanczos-based Krylov subspace methods and MEO, we obtain an operation complexity guarantee of RTR.
\begin{corollary}[Operation complexity of RTR]\label{cor:tr-oc}
	\cref{alg:tr,alg:tr-sub} finds an $(\epsilon_{g},\epsilon_{H})$-approximate second-order stationary point with the following worst-case operation complexity:
	$$
		\widetilde{O}\left(\max\left\{
		\epsilon_{H}^{-1-\tfrac{2}{\alpha}},
		\epsilon_{g}^{-\tfrac{4}{1+\alpha}}\epsilon_{H}^{-1 + \tfrac{2}{\alpha}} \right\}\cdot \epsilon_{H}^{-1 /2}\right).
	$$
	where $\widetilde{O}$ suppresses the logarithmic dependency on $\epsilon_{H}$.
	Let $\epsilon_{H} = \epsilon_{g}^{\alpha /(1+\alpha)}$, the complexity becomes
	$$
		\widetilde{O}\left(\epsilon_{g}^{- \tfrac{4+3\alpha}{2(1+\alpha)}} \right),
	$$
	which further becomes $\widetilde{O}(\epsilon_{g}^{-7/4})$ when $\alpha = 1$.
\end{corollary}

We thoroughly examine in \cref{sec:valid-troc} the applicability of \cref{lem:deltamin,lem:obj-dec-tr}, and thus \cref{cor:tr}, when employing \cref{alg:tr-sub}, making this operation complexity valid.

\begin{remark}[TCG as RTR's subproblem sovler]
	Truncated conjugate gradient (TCG) methods \citep{steihaugConjugateGradientMethod1983,toint1981efficientsparsity} constitute another practical choice for the subproblem solver in RTR \citep{zhang2023RiemannianTrust,absilTrustRegionMethodsRiemannian2007}.
	For trust region methods designed for $C^{2}$ problems in Euclidean spaces with TCG as the subproblem solver, \citet{curtis2021TrustRegionNewtonCG} provides an iteration complexity of $\widetilde{O}(\epsilon_{g}^{-3 /2})$ and an operation complexity of $\widetilde{O}(\epsilon_{g}^{-7 /4})$.
	Remarkably, with minimal adjustments, our analysis seamlessly extends to the context of RTR employing TCG as the subproblem solver, yielding an iteration complexity given in \cref{cor:tr} and an operation complexity given in \cref{cor:tr-oc}.
	This adaptability arises due to the commensurate operation complexity of TCG and Lanczos-based Krylov subspace methods. Additionally, when TCG terminates within the trust region ($\|\eta_k\|<\Delta_k$), it equates to a Krylov subspace method (with a $H_k$-conjugate basis). When $\|\eta_k\| = \Delta_k$, we only necessitate the Cauchy condition \cref{eq:tc-c}, a criterion also satisfied by TCG.
	Thus, our analysis aptly extends to encompass TCG as well.
\end{remark}

We conclude with some additional insights into the unified view of the subproblem procedures for RAR and RTR.
When \texttt{max\_flag = true}, indicating that the subproblem solver falls short of meeting the termination criteria within $\widetilde{O}(\epsilon_{H}^{-1 /2})$ iterations, both methods gain the knowledge about the approximate positive definiteness of $H_k$: RAR can deduce that $H_k \succeq -\epsilon_{H}I$, while RTR can assert that $H_k \not\succeq -\epsilon_{H}I$. However, despite this information, obtaining a suitable iteration step remains elusive at this point for both methods. Consequently, an additional process is required by both RAR and RTR to determine $\eta_k$: RAR employs a regularization on the Hessian, and RTR can leverage an MEO.
Conversely, when \texttt{max\_flag = false} and $\|g_k\|\le \epsilon_{g}$, both procedures still lack information regarding the approximate positive definiteness of $H_k$; thus, both depend on an MEO to test the second-order stationarity. However, since RAR can find a suitable iteration point satisfying the second-order termination criterion \cref{eq:tc-2} or \cref{eq:tc-d}, it does not rely on the MEO to return $\eta_k$.
{Nonetheless, we can develop a scheme similar to \cref{alg:tr-sub} for RAR, which utilizes an MEO to return $\eta_k$ (see, e.g., \cite[Lemma 7]{xu2020Newtontypemethods}). We remark that an MEO in \cref{sec:meo} enjoys the same operation complexity as a Krylov subspace method with a perturbed subspace basis.}

\section{Computational Experiments} \label{sec:exp}

We implement the proposed modified RAR and RTR algorithms, and compare the performance of their \textit{$\mu$-aware} and \textit{$\mu$-agnostic} variants.
Recall that $\mu$ is the \holder continuity order of the objective function's Hessian.
For $\mu$-aware methods, we set the subproblem inexactness parameter $\theta = \mu$, while for $\mu$-agnostic methods, we set $\theta = 1$.
Additionally, for $\mu$-aware RAR, we set the regularization order $\omega = \mu$; for $\mu$-agnostic RAR (i.e., Riemannian ARC), we set $\omega = 1$.
Our implmentation is built upon the \texttt{manopt} package,\footnote{\href{https://www.manopt.org}{https://www.manopt.org}} which includes the Riemannian trust-region and ARC framework, as well as Lanczos-based Krylov subspace methods and MEO needed in our subproblem solvers (\cref{alg:ar-sub,alg:tr-sub}).

We evaluate our methods on a quadratic minimization problem with general-order distance regularization on the sphere $\mathbb{S}^{n}$:
\[
	\min_{p \in \mathbb{S}^{n}} \quad f(p) = f_{1}(p) + \dist(p,b)^{2+\mu}
	,\]
where $f_1(p) = c\left< p,Ap \right>_{E}$ is smooth but not necessarily convex, and $\dist(p,b)^{2+\mu}\in C^{2,\mu}$.
This problem is a variant of constrained quadratic maximization problems commonly encountered in applications such as game intervention theory \citep{parise2023GraphonGames,kor2023Multiactivityinfluencea}.
Further details on the connection to these applications and the smoothness properties of the objective can be found in \cref{apx:exp}.

\cref{fig:illus} presents an example where different regularization orders lead to different convergence points $\hat{p}$ solved by RTR. The colorbar indicates the function value of $f_1$ over the sphere $\mathbb{S}^{2}$. We annotate the convergence points with their objective values $(f_1(\hat{p}), \dist(\hat{p},b)^{2+\mu})$ and refer to $b$ as the base point. With weaker regularization ($\mu = 0.8$), the solution stays near the region with a lower $f_1$ value, whereas stronger regularization ($\mu = 0.2$) pulls the solution closer to the base point $b$. This concrete example illustrates how the regularization order---and thus the smoothness of the objective---can significantly influence the solution.

\begin{figure}[th]
	\centering
	\includegraphics[width=0.55\textwidth]{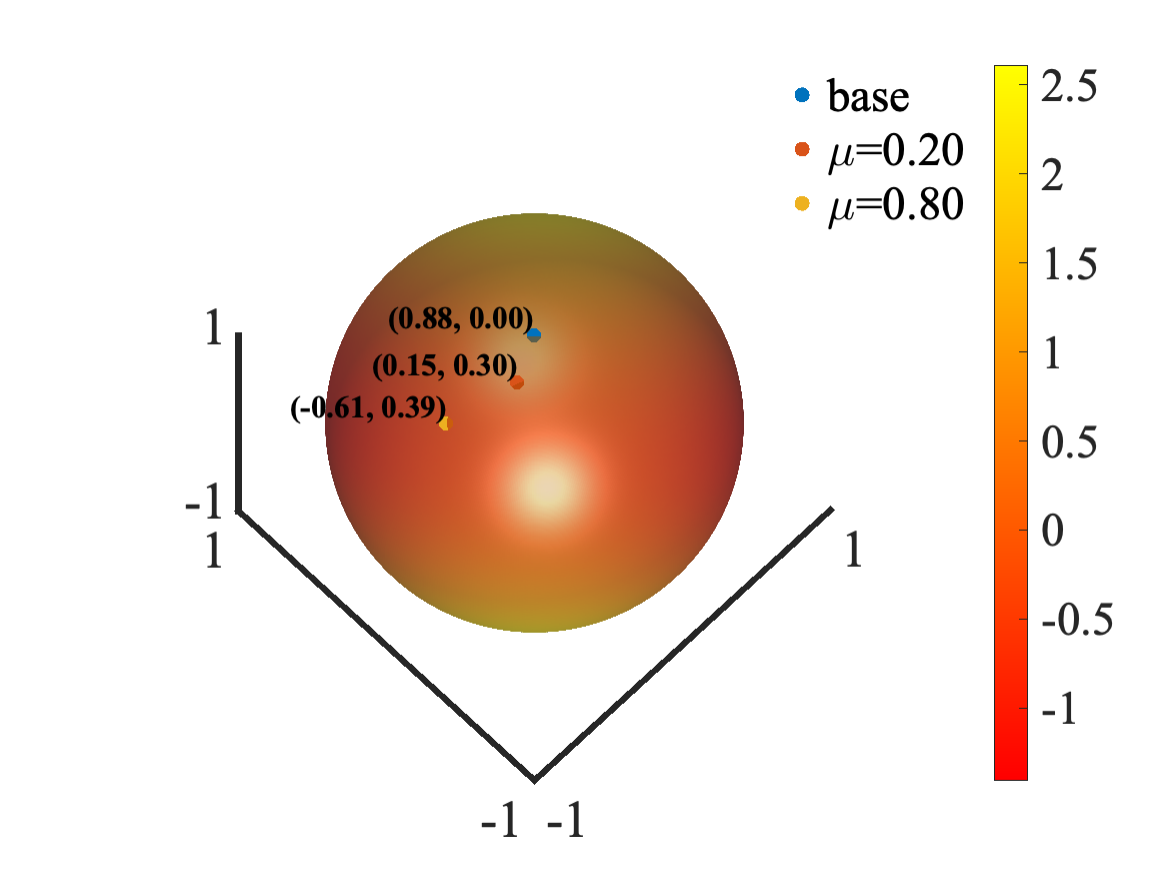}
	\caption{~Illustration of solutions for different regularization orders. In this example, $n=2$, $c=1$, $A = Z+Z^{T}$ with $Z \sim \mathcal{N}_{3\times 3}(0,1)$, and $b = (1,0,0)$ (labeled as the base point). The colorbar indicates the function value of $f_1$ on the sphere $\mathbb{S}^{2}$. In the text label of each point, the first component is the function value of $f_1$ and the second component is the distance regularization.}
	\label{fig:illus}
\end{figure}

In our experiments, we set $n=10^3$, $c=1$, $A=Z+Z^{T}$ where $Z \sim \mathcal{N}_{(n+1)\times (n+1)}(0,1)$, and $b = (1,0,\ldots,0)^{T}$.
Results are reported with a 95\% confidence region, obtained over 10 independent runs with randomly generated problem instances and initial points.
All other parameters follow the default settings in \texttt{manopt}, shared across all methods.
The stopping criteria target a second-order stationary point with tolorances $\epsilon_{g} = \|\grad f(x_0)\|_2\cdot 10^{-14}$ and $\epsilon_{H} = \epsilon_{g}^{\frac{1}{2}}$.
Since the synthetic problem is relatively easy, all methods have a similar iteration complexity under 100 across all problem instances.
Therefore, we focus on comparing performance in terms of operation complexity (number of Hessian-vector products) and runtime.
We also observe that in most instances, the MEO in our subproblem solvers is invoked only once at the end to verify second-order stationarity.
These observations are consistent with the findings in \cite{curtis2021TrustRegionNewtonCG}.

\cref{fig:rar,fig:rtr} report the performance improvement of $\mu$-aware versions of RAR and RTR over their $\mu$-agnostic counterparts.
For the convergence performance comparison, we plot the results for $\mu=0.4$. The results for other values of $\mu$ are similar and consistent with the results of operation complexity improvement.
Please see also the full results reported in \cref{tab:comp,tab:conv}.
We observe universal improvements from incorporating $\mu$-awareness across all settings.
For both RAR and RTR, the improvments are more significant when $\mu\approx 0.4$.
Overall, RAR enjoys greater improvements than RTR, potentially due to its further $\mu$-awareness in the regularization order $\omega=\mu$, in addition to the inexactness parameter $\theta=\mu$.

As $\mu$ increases beyond 0.4, the improvement from $\mu$-awareness diminishes monotonically, with no improvement at $\mu = 1$. This aligns with both our theoretical analysis of operation complexity and the intuition that as the Hessian approaches Lipschitz continuous, the distinction between $\mu$-aware and $\mu$-agnostic methods vanishes.
On the other hand, for $\mu < 0.4$, the improvement is less pronounced, potentially because the objective function lacks sufficient smoothness for higher-order methods to be effective. This low-$\mu$ regime may require further tailored algorithms to realize more substantial performance gains.

\begin{figure}[th]
	\centering
	\begin{subfigure}{0.45\textwidth}
		\centering
		\includegraphics[width=0.8\textwidth]{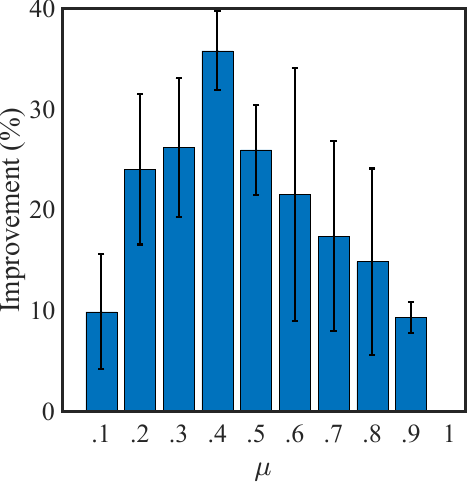}
		\caption{Operation complexity Improvement of $\mu$-aware RAR.} \label{fig:rarcomp}
	\end{subfigure}
	\begin{subfigure}{0.45\textwidth}
		\centering
		\includegraphics[width=0.86\textwidth]{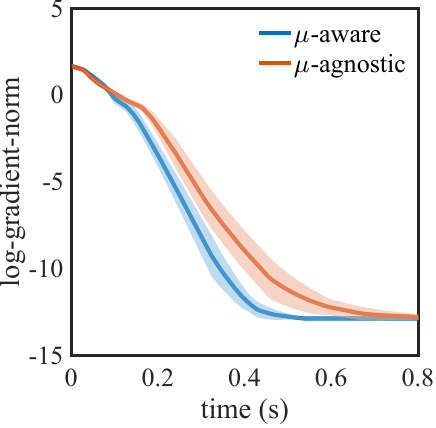}
		\vspace{-3.5pt}
		\caption{Convergence comparison for $\mu=0.4$.} \label{fig:rarconv}
	\end{subfigure}
	\caption{~Comparison of $\mu$-aware and $\mu$-agnostic RAR.}\label{fig:rar}
\end{figure}

\begin{figure}[th]
	\centering
	\begin{subfigure}{0.45\textwidth}
		\centering
		\includegraphics[width=0.8\textwidth]{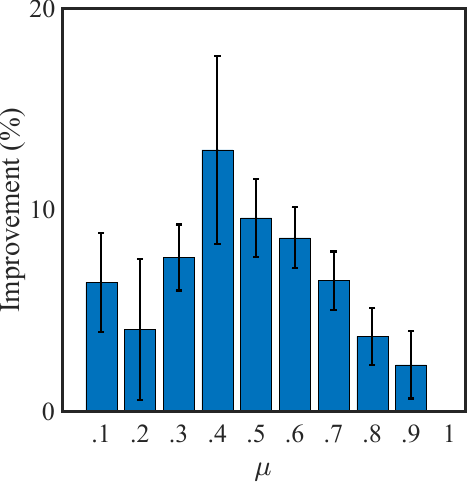}
		\caption{Operation complexity improvement of $\mu$-aware RTR.} \label{fig:rtrcomp}
	\end{subfigure}
	\begin{subfigure}{0.45\textwidth}
		\centering
		\includegraphics[width=0.86\textwidth]{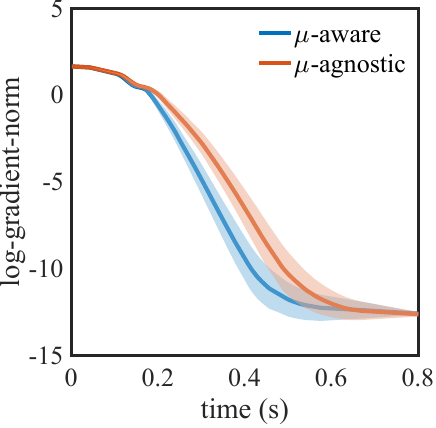}
		\vspace{-3.5pt}
		\caption{Convergence comparison for $\mu=0.4$.} \label{fig:rtrconv}
	\end{subfigure}
	\caption{~Comparison of $\mu$-aware and $\mu$-agnostic RTR.}\label{fig:rtr}
\end{figure}

\begin{table}[th]
	\caption{~Average operation complexity comparison.}
	\label{tab:comp}
	\begin{tabular}{@{}ccrlllllllll@{}}\toprule
		\textbf{Method}                          & \textbf{$\mu$-awareness} & \textbf{$\mu\!=$0.1} & \textbf{0.2} & \textbf{0.3} & \textbf{0.4} & \textbf{0.5} & \textbf{0.6} & \textbf{0.7} & \textbf{0.8} & \textbf{0.9} & \textbf{1.0} \\ \midrule
		\multicolumn{1}{c}{\multirow{2}{*}{RAR}} & \cmark                   & 361                  & 345          & 294          & 268          & 321          & 350          & 321          & 318          & 364          & 400          \\
		\multicolumn{1}{c}{}                     & \xmark                   & 401                  & 462          & 404          & 421          & 438          & 467          & 383          & 375          & 401          & 400          \\
		\multicolumn{2}{c}{improvement (\%)}     & 9.8                      & 24.1                 & 26.2         & 35.8         & 26.0         & 21.6         & 17.4         & 14.9         & 9.3          & 0                           \\ \midrule
		\multicolumn{1}{c}{\multirow{2}{*}{RTR}} & \cmark                   & 278                  & 298          & 302          & 293          & 279          & 293          & 293          & 323          & 296          & 323          \\
		\multicolumn{1}{c}{}                     & \xmark                   & 297                  & 309          & 327          & 341          & 309          & 321          & 313          & 336          & 303          & 323          \\
		\multicolumn{2}{c}{improvement (\%)}     & 6.4                      & 4.1                  & 7.6          & 13.0         & 9.6          & 8.6          & 6.5          & 3.7          & 2.3          & 0                           \\\bottomrule
	\end{tabular}
\end{table}

\begin{table}[th]
	\caption{~Average runtime comparison.}
	\label{tab:conv}
	\begin{tabular}{@{}ccrlllllllll@{}}\toprule
		\textbf{Method}                          & \textbf{$\mu$-awareness} & \textbf{$\mu\!=$0.1} & \textbf{0.2} & \textbf{0.3} & \textbf{0.4} & \textbf{0.5} & \textbf{0.6} & \textbf{0.7} & \textbf{0.8} & \textbf{0.9} & \textbf{1.0} \\ \midrule
		\multicolumn{1}{c}{\multirow{2}{*}{RAR}} & \cmark                   & .6610                & .6204        & .5645        & .4637        & .5472        & .8731        & .8205        & .7363        & .6811        & .6959        \\
		\multicolumn{1}{c}{}                     & \xmark                   & .6892                & .8446        & .7405        & .7062        & .7246        & 1.2308       & .9377        & .7685        & .7543        & .7138        \\
		\multicolumn{2}{c}{improvement (\%)}     & 4.1                      & 26.6                 & 23.8         & 34.3         & 24.5         & 29.1         & 12.5         & 4.2          & 9.7          & 2.5                         \\ \midrule
		\multicolumn{1}{c}{\multirow{2}{*}{RTR}} & \cmark                   & .4717                & .5113        & .4880        & .4766        & .4776        & .4681        & .4673        & .6035        & .4999        & .5245        \\
		\multicolumn{1}{c}{}                     & \xmark                   & .5203                & .5316        & .5458        & .5586        & .5045        & .5111        & .5276        & .6596        & .5215        & .5059        \\
		\multicolumn{2}{c}{improvement (\%)}     & 9.3                      & 3.8                  & 10.6         & 14.7         & 5.3          & 8.4          & 11.4         & 8.5          & 4.1          & -3.7                        \\ \bottomrule
	\end{tabular}
\end{table}

\ifSubfilesClassLoaded{\bibliography{ic}}{}

\bibliography{ic.bib}

\renewcommand \thepart{}
\renewcommand \partname{}

\appendix
\part*{Appendix}
\renewcommand{\thesection}{\Alph{section}}
\section{Proof of Propositions} 

\subsection{Proof of Proposition~{\ref{lem:ma-diff}}}
\begin{proof}
	The first inequality is given by \citet[Proposition 3]{zhang2023RiemannianTrust} and the second is by \citet[Corollary 5.1]{zhang2023RiemannianTrust}.
	In this proof, we omit the subscript $x_k$ of $R_{x_k}$ and $\exp _{x_k}$.
	By the Taylor expansion on manifolds {(\cite[Section 4.1]{boumal2023introductionoptimization})}, there exists $\tau\in[0,1]$ such that
	$$
		f(R(\eta_{k})) = f(\exp (\eta_{k})) + \left<\g f(\gamma(\tau)), \gamma'(\tau) \right>,
	$$
	where $\gamma$ is the geodesic from $\exp (\eta_k)$ to $R(\eta_k)$. Then, we have
	\begin{align}\notag
		    & |f(R(\eta_k)) - f(\exp (\eta_k))|                                                                        \\\notag
		=   & \left|\left<\g f(\gamma(\tau)), \gamma'(\tau) \right>\right|                                             \\\label{eq:ma-diff-1}
		=   & \left| \left<P_\gamma^{\tau\to 0}\g f(\gamma(\tau)), \gamma'(0)\right> \right|                           \\\notag
		\le & \|P_\gamma^{\tau\to 0}\g f(\gamma(\tau))\|\|\gamma'(0)\|                                                 \\\label{eq:ma-diff-2}
		=   & \|\g f(x_k) - \g f(x_k) + P_\gamma^{\tau\to 0}\g f(\gamma(\tau))\| \cdot \dist (\exp (\eta_k),R(\eta_k)) \\\label{eq:ma-diff-3}
		\le & \left( \|\g f(x_k)\| + \beta_H \dist(x_k,\gamma(\tau)) \right)\cdot \dist(\exp (\eta_k),R(\eta_k)),
	\end{align}
	where \cref{eq:ma-diff-1} use the fact that the parallel transport $P_{\gamma}^{\tau\to 0}$ preserves inner product and $P_{\gamma}^{\tau\to 0}\gamma'(\tau) = \gamma(0)$ becomes $\gamma'$ is parallel along $\gamma$;
	\cref{eq:ma-diff-2} is by the definition of a geodesic {in \cref{sec:pre}};
	and in \cref{eq:ma-diff-3}, $\beta_H$ is the uniform operator norm bound of $\hess f$, and then $\g f$ is $\beta_H$-Lipschitz continuous.
	Then by the triangle inequality, we have
	\begin{equation}\label{eq:ma-diff-4}
		\dist(x_k,\gamma(\tau))
		\le \dist(x_k,\exp (\eta_k)) + \dist(\exp (\eta_k),\gamma(\tau))
		\le \|\eta_k\| + \dist(\exp (\eta_k), R(\eta_k)).
	\end{equation}
	Combining \eqref{eq:ma-diff-3}, \eqref{eq:ma-diff-4}, and the first inequality of \cref{lem:ma-diff} gives the desired result.
\end{proof}

\subsection{Proof of Proposition~\ref{lem:tc}}
\begin{proof}
	We denote $r_k \coloneqq \grad \bar{m}_{x_k}(\eta_k)$ and $s_k \coloneqq \g\phi(\eta_k;\sigma_k)$.
	Then, $r_k = g_k + H_{k}\eta_k + s_k$.
	For $k\in \mathcal{S}$, we have
	\begin{align}\notag
		\|g_{k+1}\| & = \|g_{k+1} + P_{k,k+1}(r_{k} - g_k - H_{k}\eta_k- s_k)\|                             \\\notag
		            & \le \|g_{k+1} - P_{k,k+1}(g_{k} + H_{k}\eta_k)\| + \|r_{k}\| + \|s_k\|                \\\notag
		            & \le \underbrace{\|g_{k+1} - P_{k,k+1}(g_k\!+\!H_{k}\exp_{x_k}^{-1}(x_{k+1}))\|}_{S_1}
		+\underbrace{\|H_k(\exp _{x_k}^{-1}(x_{k+1}) - \eta_k)\|}_{S_2}
		+\|\eta_k\|^{1+\theta}
		+ \|s_k\|,
	\end{align}
	where we write $P_{k,k+1} \coloneqq P_{x_{k},x_{k+1}}$, which is norm preserving, for notational simplicity and use \cref{eq:tc-1}: $\|r_k\| \le\|\eta_k\|^{1+\theta}$.
	Since $f\in C^{2,\mu}$, by the Taylor expansion of $\g f(x_k)$ (\cite[Section 4.1]{boumal2023introductionoptimization}), we have
	\begin{align}\notag
		S_1 \le & C_{H}\dist(x_k,x_{k+1})^{1+\mu}                                                                            \\\notag
		\le     & C_{H}(\dist(x_k, \exp_{x_k}(\eta_k)) + \dist(\exp _{x_k}(\eta_k), R_{x_k}(\eta_{k})))^{1+\mu}              \\\label{eq:s-dec-2}
		\le     & 2C_{H}\dist(x_k, \exp_{x_k}(\eta_k))^{1+\mu} + 2C_{H}\dist(\exp _{x_k}(\eta_k), R_{x_k}(\eta_{k}))^{1+\mu} \\\label{eq:s-dec-3}
		\le     & 2C_{H}\|\eta_k\|^{1+\mu} + 2C_{H}C_{R}^{1+\mu}\|\eta_k\|^{(1+\mu)(1+\nu)},
	\end{align}
	where \eqref{eq:s-dec-2} is by the convexity of function $z\mapsto z^{1+\mu}$ and \eqref{eq:s-dec-3} is by \cref{lem:ma-diff}, which also gives
	$$
		S_2 \le \beta_H C_{R}\|\eta_k\|^{1+\nu}.
	$$
	Substituting $S_1$ and $S_2$ with the above bounds gives the first result of \cref{lem:tc}.

	For $\eta_k$ satisfying termination criterion \cref{eq:tc-2}, we have
	$$
		\begin{aligned}
			-\lambda_{\min}(H_k) = & \lambda_{\max}(-H_k) = \lambda_{\max}(-\hess \bar{m}(\eta_k) + \hess \phi(\eta_k;\sigma_k)) \\
			\le                    & -\lambda_{\min}(\hess \bar{m}(\eta_k)) + \lambda_{\max}(\hess \phi(\eta_k;\sigma_k))        \\
			\le                    & \|\eta_k\|^{\theta} + \lambda_{\max}(\hess \phi(\eta_k;\sigma_k)),
		\end{aligned}
	$$
	where the first inequality is by \cite[Theorem 10.21]{zhang2011Matrixtheory} and uses the fact that $\hess \bar{m}$ and $\hess \phi$ are Hermitian.
\end{proof}

\subsection{Verifying Assumption~\ref{lem:g-bound}} \label{sec:pf-g-bound}

{We claim that the solution returned by the Krylov subspace methods satisfies \cref{lem:g-bound}}.
Let $\xi_1$ be the minimizer of \cref{eq:model-ar} restricted in the first Krylov subspace $\spa\{g_k\}$.
Then we have $\grad_{\tau} \bar{m}(\tau g_k) = \tau \left<g_k, \grad \bar{m}(\xi_1)\right> = 0$, where $\xi_1 = \tau g_k$. Calculating the gradient of $\bar{m}$ gives
\begin{equation}\label{eq:pf-g-bound-1}
	\left<g_k, g_k + H_k\xi_1 + \grad \phi(\xi_1) \right> = 0,
\end{equation}
which further gives
$$
	\|g_k\|^2 = -\left<g_k,H_k\xi_1 \right> - \left<g_k, \g \phi(\xi_1) \right> \le \|g_k\|\|H_k\|\|\xi_1\| + \|g_k\|\|\g \phi(\xi_1)\|.
$$
Then for RAR, since $\|\grad \phi(\xi)\| = \sigma_k\|\xi\|^{1+\omega}$ increases with the magnitude of $\xi$ and $\|\xi_1\| \le \|\eta_k\|$ \citep[Theorem 1]{cartis2020monotonicestimates}, we have
\begin{equation}\label{eq:pf-g-bound-2}
	\|g_k\| \le \beta_{H}\|\xi_1\| + \|\g \phi(\xi_1)\| \le \beta_{H}\|\eta_k\| + \|\g \phi(\eta_k)\|.
\end{equation}
For RTR, $\|\xi_1\| \le \|\eta_k\|$ also holds
\citep{steihaugConjugateGradientMethod1983,luksan2008Lagrangemultipliers}. Therefore, if $\|\eta_k\|<\Delta_k$, we know $\|\xi_1\| < \Delta_k$, and then \cref{eq:pf-g-bound-1} still holds.
Since $\|\g \phi(\xi)\| {= \epsilon_{H}\|\xi\| /2}$ also increases with the magnitude of $\xi$ for RTR, \cref{eq:pf-g-bound-2} still holds.

\ifSubfilesClassLoaded{\bibliography{ic}}{}

\subsection{Proof of Proposition~\ref{prop:s}} \label{sec:pf-s}

\begin{proof}
	In this proof, we omit the superscript and subscript of $\bar{m}^{\mathrm{ar}}_{x_k}$.
	By the Cauchy termination criterion \cref{eq:tc-c}, we have
	\begin{equation}\label{eq:prop-s-2}
		0 \le \bar{m}(0) - \bar{m}(\eta_k) \le \|g_k\|\|\eta_k\| + \frac{1}{2}\beta_{H}\|\eta_k\|^2 - \frac{\sigma_k}{2+\omega}\|\eta_k\|^{2+\omega}.
	\end{equation}
	If $\|\eta_k\|\ge \|g_k\|$, \cref{eq:prop-s-2} gives
	$$
		\|\eta_k\|^2 \cdot \left( 1+\frac{\beta_{H}}{2} - \frac{\sigma_k}{2+\omega}\|\eta_k\|^{\omega} \right) \ge 0.
	$$
	Therefore, we have
	\begin{equation}\label{eq:prop-s-3}
		\|\eta_k\| \le \left( \frac{(2+\beta_{H})(2+\omega)}{2\sigma_k} \right) ^{1 /\omega}\vee \|g_k\| .
	\end{equation}
	Another decomposition of the right hand side \cref{eq:prop-s-2} gives
	$$
		0 \le \|\eta_k\|\cdot \left( \|g_k\| - \frac{\sigma_k}{2(2+\omega)}\|\eta_k\|^{1+\omega} \right) + \|\eta_k\|^{2}\cdot \left( \frac{\beta_{H}}{2} - \frac{\sigma_k}{2(2+\omega)}\|\eta_k\|^{\omega} \right).
	$$
	Hence the two terms on the right hand side of the above inequality cannot both be negative. This gives
	\begin{equation}\label{eq:prop-s-4}
		\|\eta_k\| \le \left( \frac{2\beta_{H}(2+\omega)}{2\sigma_k} \right) ^{1 /\omega}\vee \left( \frac{2(2+\omega)\|g_k\|}{\sigma_k} \right) ^{1 /(1+\omega)}.
	\end{equation}
	Utilizing $\omega\in (0,1]$ and the fact that $(a \vee b)\w(a\vee c) = a\vee (b \w c)$, \cref{eq:prop-s-3,eq:prop-s-4} give
	$$
		\|\eta_k\| \le \left(\frac{3(\beta_H+1)}{\sigma_k}\right)^{1 /\omega} \vee \left( \|g_k\| \w \left( \frac{6\|g_k\|}{\sigma_k} \right)^{1 /(1+\omega)}  \right).
	$$
\end{proof}

\subsection{Proof of Proposition~\ref{prop:g}} \label{sec:pf-g-bound-uni}

Before proving \cref{prop:g}, we provide a lemma on the lower bound of the model decrease.
\begin{lemma}[Cauchy decrease]\label{lem:cdec}
	If $\eta_k$ satisfies the Cauchy termination criterion \cref{eq:tc-c}, then we have
	$$
		\bar{m}^{\mathrm{ar}}_{x_k}(0) - \bar{m}^{\mathrm{ar}}_{x_k}(\eta_k) \ge \frac{\|g_k\|^2}{4\left( \beta_{H}\vee \left( \sigma_{k}^{1 /(1+\omega)}\|g_k\|^{\omega /(1+\omega)} \right)  \right) }.
	$$
\end{lemma}
\begin{proof}
	In this proof, we omit the superscript and subscript of $\bar{m}^{\mathrm{ar}}_{x_k}$.
	Let $\eta^{\mathrm{C}}$ be the Cauchy point. Then for any $\tau\in\R$, we have $\bar{m}(\eta^{\mathrm{C}}) \le \bar{m}(-\tau g_k)$.
	Therefore, we have
	$$
		\begin{aligned}
			\bar{m}(0)-\bar{m}(\eta^\mathrm{C})
			 & \ge \bar{m}(0)-\bar{m}\left(-\tau g_k\right)                                                                                                              \\
			 & =\tau\left\|g_k\right\|^2-\frac{1}{2} \tau^2\left\langle g_k, H_k g_k\right\rangle-\frac{\sigma_k}{2+\omega} \tau^{2+\omega}\left\|g_k\right\|^{2+\omega} \\
			 & \ge \tau\left\|g_k\right\|^2\left(1-\frac{\tau \beta_H}{2}-\frac{\sigma_k \tau^{1+\omega}\left\|g_k\right\|^{\omega}}{2+\omega}\right).
		\end{aligned}
	$$
	The above inequality holds for any $\tau$. Let $\tau=\frac{1}{2 \beta_H} \vee\left(\frac{2+\omega}{4 \sigma_k\left\|g_k\right\|^\omega}\right)^{1 /(1+\omega)}$. We get
	\begin{align}\notag
		\bar{m}(0)-\bar{m}\left(\eta^\mathrm{C}\right)
		 & \geqslant \tau\left\|g_k\right\|^2\left(\left(\frac{1}{2}-\frac{1}{2 \beta_H} \cdot \frac{\beta_H}{2}\right)
		+\left(\frac{1}{2}-\frac{2+\omega}{4 \sigma_k\left\|g_k\right\|^{\omega}}\cdot \frac{\sigma_k\left\|g_k\right\|^{\omega}}{2+\omega}\right)\right)     \\\notag
		 & =\frac{\tau\left\|g_k\right\|^2}{2}                                                                                                                \\\notag
		 & \geqslant \frac{\left\|g_k\right\|^2}{4\left(\beta_{H} \vee \left(\sigma_k^{1/(1+\omega)} \left\|g_k\right\|^{\omega /(1+ \omega)}\right)\right)}.
	\end{align}
	Notice that the above inequality also holds for $g_k=0$, in which case we let $\tau = 1 /(2\beta_{H})$.
\end{proof}

\begin{proof}[Proof of \cref{prop:g}]
	Our proof follows the same logic of \cite[Theorem 2.5 and Corollary 2.6]{cartis2011Adaptivecubica} which proves that $\lim_{k}\|g_k\|=0$.
	Here we aim to show $\limsup_k \|g_k\| < +\infty$.
	The result automatically holds if $\mathcal{S}$ is finite. Thus, we assume $|\mathcal{S}| = +\infty$ in the rest of the proof.
	We first claim that $\liminf_{k}\|g_k\| < +\infty$.
	If not, for any $C>0$, there exists $K_1 \in \mathbb{N}$ such that for any $k\ge K_1$, it holds that $\|g_k\| > C$.
	By \cref{lem:cdec}, we have
	$$
		\sum_{k\in \mathcal{S}}f(x_k) - f(x_{k+1})
		\ge \sum_{\substack{k\in \mathcal{S},~ k \ge K_1}} \frac{\varrho_{1}\|g_k\|^2}{4\left( \beta_{H}\vee \left( \sigma_k ^{1 /(1+\omega)}\|g_k\|^{\omega /(1+\omega)} \right)  \right) }.
	$$
	Since $f$ is bounded below, we know the summand sequence of the right hand side of the above inequality converges to zero.
	Then since $\|g_k\| > C$, we know that $\beta_{H}$ in the denominator must be inactive when $k\ge K_2$ for some $K_2 \ge K_1$.
	Therefore, we get
	$$
		+\infty > \sum_{k\in \mathcal{S}}f_k - f_{k+1} \ge \sum_{\substack{k\in \mathcal{S},~k\ge K_2}} \frac{\varrho_1\|g_k\|^{2 - \omega /(1+\omega)}}{4\sigma_{k}^{1 /(1+\omega)}}
		\ge \frac{\varrho_1 C}{4}\sum_{\substack{k\in \mathcal{S},~k\ge K_2}}\left( \frac{\|g_k\|}{\sigma_k} \right) ^{1 /(1+\omega)},
	$$
	which gives
	\begin{equation}\label{eq:prop-g-1}
		\lim_{k_1 \to \infty}
		\sum_{\substack{k\in \mathcal{S},~ k\ge k_1}}\left( \frac{\|g_k\|}{\sigma_k} \right) ^{1 /(1+\omega)} = 0,
	\end{equation}
	which further gives $\lim_{\mathcal{S}\ni k \to \infty} \|g_k\| /\sigma_k = 0$ and thus $\lim_{\mathcal{S}\ni k \to \infty}\sigma_k = +\infty$ {due to $\|g_k\|>C$}.
	Let $C \ge \beta_{H}+1$. Then, by \cref{prop:s}, there exists $K_3 \ge K_2$ such that for any $k\in \mathcal{S}$ and $k\ge K_3$, we have
	\begin{equation}\label{eq:prop-g-2}
		\|\eta_k\| \le \left( \frac{6\|g_k\|}{\sigma_k} \right) ^{1 /(1+\omega)} \le 1.
	\end{equation}
	Then, by the triangle inequality and \cref{lem:ma-diff}, for any $k_{2} \ge k_1 \ge K_3$, we have
	\begin{align}\notag
		\dist(x_{k_2},x_{k_1})
		 & \le\sum_{\substack{k\in \mathcal{S}                                                                                           \\k_1\le k{<}k_2}}\dist(x_k,x_{k+1})\\\notag
		 & = \sum_{\substack{k\in \mathcal{S}                                                                                            \\k_1\le k{<}k_2}}\|\exp _{x_k}^{-1}(R_{x_k}(\eta_k))\|\\\notag
		 & \hspace{-5pt} \stackrel{\hyperref[lem:ma-diff]{\scriptscriptstyle\textup{Prop.\,1}}}{\le}\hspace{-8pt}
		\sum_{\substack{k\in \mathcal{S}                                                                                                 \\k_1\le k{<}k_2}}\|\eta_k\| \cdot (1 + C_{R}\|\eta_k\|^{\nu})\\\notag
		 & \stackrel{\cref{eq:prop-g-2}}{\le}                         (1 + C_{R}) \cdot \sum_{\substack{k\in \mathcal{S}                 \\k_1\le k<k_2}}\|\eta_k\|\\\notag
		 & \stackrel{\cref{eq:prop-g-2}}{\le}                         6^{1 /(1+\omega)}(1 + C_{R})\cdot \sum_{\substack{k\in \mathcal{S} \\k_1\le k<k_2}} \left( \frac{\|g_k\|}{\sigma_k} \right)^{1 /(1+\omega)}\\\notag
		 & \le 6^{1 /(1+\omega)}(1 + C_{R})\cdot \sum_{\substack{k\in \mathcal{S}                                                        \\k\ge k_1}} \left( \frac{\|g_k\|}{\sigma_k} \right)^{1 /(1+\omega)}\\\label{eq:prop-g-3}
		 & \stackrel{\eqref{eq:prop-g-1}}{\to} \, 0 \quad \text{as}\quad k_1 \to +\infty.
	\end{align}
	Therefore, $\{ x_k \}$ is a Cauchy sequence. Since $\M$ is complete, we know $\{ x_k \}$ converges, contradicting the hypothesis $\liminf_k \|g_k\| = +\infty$.

	Now suppose $\liminf_{k}\|g_k\| \le C$. If {$\limsup_{k}\|g_k\|=+\infty$}, then there exists $\mathcal{S}_{C}\subset \mathcal{S}$ such that $|\mathcal{S}_{C}| = +\infty$ and $\|g_k\| \ge C+1$ for any $k\in \mathcal{S}_{C}$.
	For any $k_1\in \mathcal{S}_{C}$, let $k_2$ be the smallest integer such that $k_1 \le k_2 \in \mathcal{S}$ and $\|g_{k_2}\| \le C$. Consider the infinite index set formed by these consecutive index sequences (say $\mathcal J$), whose definition implies $\|g_k\| > C$ for all $k\in \mathcal J$.
	We observe a similar pattern as in \cref{eq:prop-g-1,eq:prop-g-2,eq:prop-g-3}, which only relies on the condition that $\|g_k\| > C$ within the specified index set. This leads us to the conclusion that
	$$
		\dist(x_{k_2},x_{k_1}) \to 0 \quad \text{as} \quad k_1 \to +\infty.
	$$
	By the Lipschitzness of $\g f$, we get
	$$
		\|g_{k_1}\| \le \|g_{k_2}\| + \|g_{k_1} - g_{k_2}\| \to \|g_{k_2}\| \le C \quad \text{as} \quad k_1 \to +\infty,
	$$
	contradicting to the assumption that $\|g_{k_1}\|\ge C+1$. Therefore, we conclude $\limsup_{k} \|g_k\| < +\infty$ and thus $\{ g_k \}$ is bounded.
\end{proof}

\subsection{Remark on Proposition~\ref{prop:krylov}} \label{sec:prop-krylov}

{With a regular Krylov subspace basis specified in \cref{sec:kry},}
\cref{prop:krylov-ar} in \cref{prop:krylov} is directly given by \citet[Corollary 4.2]{carmon2020FirstOrderMethods} and the fact that $|\lambda_{\max}(H_k)| \vee |\lambda_{\min}(H_k)| \le \beta_{H}$, {with constant $C_{\mathrm{sub},1}$ being
\[
	C_{\mathrm{sub},1} = \sqrt{\beta_{H}\left( 8 + \log^{2}\frac{2\|g_k\|}{\left<g_k,\nu_{k}\right>} \right)}
	,\]
where $\nu_{k}$ is the corresponding eigenvector of $\lambda_{\min}(H_k)$.
Remarkably, if $\left< g_k,\nu_{k} \right> = 0$, we can perturb the basis of the Krylov subspace to ``randomize away the hard case'' (please refer to \cite[Section 5.2]{carmon2020FirstOrderMethods} for more details).
Then, $\eta_k$ is an $\epsilon$-optimal solution to \cref{eq:model-arn-2} with probability at least $1-\delta$ and $C_{\mathrm{sub},1}$ becoming
\[
	2\sqrt{\beta_{H}\left( 2 + \log^{2}(2\sqrt{n}/\delta) \right)}
	,\]
where $n$ is the dimension of the manifold.
}

{
\cref{prop:krylov-ar-2,prop:krylov-tr} are given by \cite[Corollary 4.2]{carmon2020FirstOrderMethods} and \cite[Theorem 4.1]{carmon2020FirstOrderMethods}, respectively.
Using the fact that $\frac{a + c}{b + c} > \frac{a}{c}$ for $a,b,c>0$ and $a < b$, we have
\[
C_{\mathrm{sub},2} \le \frac{\sqrt{\beta_{H}}}{4}(1 + \log (1 \vee 36\left( \bar{m}_{x_k}(0) - \bar{m}_{x_k}(\eta^{*} )\right)))
,\]
}%
where $\eta^* \coloneqq \argmin_{\eta} \bar{m}_{x_k}(\eta)$. Recall that we set $\alpha = \mu = \nu = \theta = \omega$.

For RAR (\cref{prop:krylov-ar-2}), by \cref{prop:s} (which remains valid when using \cref{alg:ar-sub} as the subproblem solver, as discussed in \cref{sec:valid-aroc}), we have
$$
	\|\eta^*\| \le \left( \frac{3(\beta_{H}+1)}{\ubar{\sigma}} \right) ^{1 /\alpha} \vee \|g_k\| = O(\|g_k\|\vee 1).
$$
Then, we get
$$
	\begin{aligned}
		\bar{m}^{\mathrm{ar}}_{x_k}(0) - \bar{m}^{\mathrm{ar}}_{x_k}(\eta^*) \le & \|g_k\|\|\eta^*\| + \frac{1}{2}\|H_k\|\|\eta^*\|^2 + \frac{\sigma_k}{2+\alpha} \|\eta^*\|^{2+\alpha}     \\
		\le                                                                      & \|g_k\|\|\eta^*\| + \frac{\beta_{H}}{2}\|\eta^*\|^2 + \frac{\bar{\sigma}}{2+\alpha}\|\eta^*\|^{2+\alpha} \\
		=                                                                        & O(\|g_k\|^{2+\alpha}\vee 1),
	\end{aligned}
$$
where we use the fact that $\bar{\sigma} = O(1)$ when $\alpha = \mu = \nu = \theta = \omega$, which is regardless of the subproblem solver (see \cref{sec:valid-aroc}).

For RTR (\cref{prop:krylov-tr}), we have
$$
	\bar{m}^{\mathrm{tr}}_{x_k}(0) - \bar{m}^{\mathrm{tr}}_{x_k}(\eta^*) \le \|g_k\|\|\eta^*\| + \frac{1}{2}(\|H_k\| + 2\epsilon_{H})\|\eta^*\|^2
	\le \|g_k\|\bar{\Delta} + \frac{\beta_{H}+2\epsilon_{H}}{2}\bar{\Delta}^2
	= O(\|g_k\|\vee 1).
$$

Then, we can proceed the analysis with the maximum number of subproblem iterations explicitly dependent on $\|g_k\|$.
Notably, both RAR and RTR converge to a \textit{small gradient region} characterized by $\|g_k\|\le \epsilon_{g}$. Consequently, we can assert the existence of a positive constant $\beta_{g}$ such that $\|g_k\|\le \beta_{g}$ for all $k\in \mathcal{K}$.
	{
		Hence, we get an explicit expression of $C_{\mathrm{sub},2}=O(1)$:
		\[
			C_{\mathrm{sub},2} \le \frac{\sqrt{\beta_{H}}}{4} \left(1 + \log \left( 36 \left(1 + \frac{\beta_{H}}{2} + \frac{\bar{\sigma}}{1+\alpha} + \bar{\Delta} + \frac{\beta_{H}+2\epsilon_{H}}{2}\bar{\Delta}^{2}\right)\left(1 \vee \beta_{g} \vee \left(\frac{3(\beta_{H}+1)}{\ubar{\sigma}}\right)^{1 /\alpha} \right)^{2+\alpha} \right)\right)
			.\]
	}

\subsection{Proof of Proposition~\ref{prop:tcd}}\label{sec:prop-tcd}

\begin{proof}
	In this proof, we omit the subscript and superscript of $\bar{m}^{\mathrm{ar}}_{x_k}$.
	By the optimality condition of $\min \bar{m}(\eta)$ \Citep[Theorem 1.1]{hsia2017theory}, we know
	$$
		\begin{cases}
			g_k + H_k \eta^* + \sigma_k\|\eta^*\|^{\alpha}\eta^* = \g \bar{m}(\eta^*) = 0, \quad                       & \text{first-order optimality},  \\
			\lambda_{\min}(H_k) + \sigma_k\|\eta^{*}\|^{\alpha} \ge \lambda_{\min}(\hess \bar{m}(\eta^*)) \ge 0, \quad & \text{second-order optimality},
		\end{cases}
	$$
	where $\eta^{*} = \argmin_{\eta} \bar{m}^{\mathrm{ar}}_{x_k}(\eta)$.
	Since $H_k \not\succeq -\epsilon_{H}I$, by the second-order optimality condition, we get $\|\eta^*\|\ge (\epsilon_{H}/\sigma_k)^{1 /\alpha}$.
	The first-order optimality condition gives
	$$
		\begin{aligned}
			\bar{m}(0) - \bar{m}(\eta^*)
			 & = -\left<g_k,\eta^* \right> - \frac{1}{2}\left<\eta^*, H_k \eta^* \right> - \frac{\sigma_k}{2+\alpha}\|\eta^*\|^{2+\alpha}                                         \\
			 & = \left<\eta^*, H_k \eta^* \right> + \sigma_k \|\eta^*\|^{2+\alpha} - \frac{1}{2}\left<\eta^*, H_k \eta^* \right> - \frac{\sigma_k}{2+\alpha}\|\eta^*\|^{2+\alpha} \\
			 & = \frac{1}{2}\left<\eta^*, H_k \eta^* \right> + \frac{1+\alpha}{2+\alpha}\sigma_k\|\eta^*\|^{2+\alpha}.
		\end{aligned}
	$$
	Combined with the second-order optimality condition, we get
	$$
		\bar{m}(0) - \bar{m}(\eta^*)
		\ge \frac{1}{2}(-\sigma_k\|\eta^*\|^{2+\alpha})+ \frac{1+\alpha}{2+\alpha}\sigma_k\|\eta^*\|^{2+\alpha}
		= \frac{\alpha}{2(2+\alpha)} \sigma_k\|\eta^*\|^{2+\alpha}.
	$$

	Then if $\bar{m}(\eta) - \bar{m}(\eta^*) \le \alpha\sigma_k\|\eta^*\|^{2+\alpha} / 12$, we have
	$$
		\begin{aligned}
			m(0) - m(\eta) \ge & \bar{m}(0) - \bar{m}(\eta)
			= \bar{m}(0) - \bar{m}(\eta^*) - (\bar{m}(\eta) - \bar{m}(\eta^*))
			\ge \frac{\alpha\sigma_k}{2(2+\alpha)}\|\eta^*\|^{2+\alpha} - \frac{\alpha\sigma_k}{12}\|\eta^*\|^{2+\alpha} \\
			\ge                & \frac{\alpha\sigma_k}{12}\|\eta^*\|^{2+\alpha}
			\ge \frac{\alpha\epsilon_{H}^{(2+\alpha) / \alpha}}{12\bar{\sigma}^{2 /\alpha}}.
		\end{aligned}
	$$
	By \cref{prop:krylov-ar} of \cref{prop:krylov} (\citep[Corollary 4.2]{carmon2020FirstOrderMethods}), to make $\bar{m}(\eta) - \bar{m}(\eta^*) \le \alpha\sigma_k\|\eta^*\|^{2+\alpha} / 12$, we need at most
	{
	\[
		C_{\mathrm{sub},1}\left(\frac{\alpha\sigma_k\|\eta^*\|^{2+\alpha}}{12}\right)^{-1 /2}\|\eta^*\| + 1 /2
		= C_{\mathrm{sub},3} \|\eta^*\|^{-\alpha /2} + 1 /2
		\le C_{\mathrm{sub},3}((\epsilon_{H} / \sigma_{k})^{1 /\alpha})^{-\alpha /2} + 1 /2
		= C_{\mathrm{sub}}^{\mathrm{ar}}\epsilon_{H}^{-1 /2}
	\]
	iterations, where
	$C_{\mathrm{sub},3}\coloneqq C_{\mathrm{sub},1} \cdot (\alpha \ubar{\sigma} / 12)^{-1 /2}$
	and $C_{\mathrm{sub}}^{\mathrm{ar}} \coloneqq C_{\mathrm{sub},3} \cdot \bar{\sigma}^{1 /2} +1 /2$.
	}
\end{proof}

{
\subsection{Proof of Proposition~\ref{prop:tr-sub}}\label{sec:prop-tr-sub}
}
\begin{proof}
	Recall that $\bar{H}_k = H_k + 2\epsilon_{H}I$. Suppose $H_k \succeq -\epsilon_{H}I$ and $\|g_k\| > \epsilon_{g}$. Then $\|\bar{H}_k\|\le \beta_{H}+2\epsilon_{H}$ and $\lambda_{\min}(\bar{H}_k) \ge \epsilon_{H}$.
	Let $\eta^{*} \coloneqq -\bar{H}_{k}^{-1}g_k$.
	We have
	\[
		\|\grad \bar{m}_{x_k}(\eta) \| = \| g_k + \bar{H}_k\eta\|
		=\|\bar{H}_{k}(\eta-\eta^{*})\|
		\le (\beta_{H}+2\epsilon_{H})\|\eta-\eta^{*}\|
		.\]
	For the other direction, we have
	\begin{align*}
		\bar{m}_{x_k}(\eta) - \bar{m}_{x_k}(\eta^{*})
		 & = \left< g_k,\eta \right> + \frac{1}{2}\left< \eta,\bar{H}_{k}\eta \right> - \left< g_k,\eta^{*} \right> - \frac{1}{2}\left< \eta^{*}, \bar{H}_{k}\eta^{*} \right> \\
		 & = \left< g_k,\eta \right> + \frac{1}{2}\left< \eta,\bar{H}_{k}\eta \right> +  \frac{1}{2}\left< \eta^{*}, \bar{H}_{k}\eta^{*} \right>                              \\
		 & = \frac{1}{2}\left< \eta-\eta^{*},\bar{H}_{k}(\eta-\eta^{*} )  \right>                                                                                             \\
		 & \ge \frac{1}{2}\lambda_{\min}(\bar{H}_{k})\|\eta-\eta^{*}\|^2                                                                                                      \\
		 & \ge \frac{\epsilon_{H}}{2}\|\eta-\eta^{*}\|^2
		.\end{align*}
	Combining the above two inequalities gives
	\[
		\|\grad \bar{m}_{x_k}(\eta)\| \le (\beta_{H}+2\epsilon_{H}) \sqrt{2\epsilon_{H}^{-1} (\bar{m}_{x_k}(\eta) - \bar{m}_{x_k}(\eta^{*} )) }
		.\]
	On the other hand, by \cref{lem:g-bound}
	\[
		\|\eta_k\| \ge (\beta_{H}+2\epsilon_{H})^{-1}\|g_k\| \ge \frac{\epsilon_{g}}{\beta_{H}+2\epsilon_{H}}
		.\]
	Therefore, \cref{eq:tc-1} is satisfied if
	\[
		(\beta_{H}+2\epsilon_{H}) \sqrt{2\epsilon_{H}^{-1} (\bar{m}_{x_k}(\eta) - \bar{m}_{x_k}(\eta^{*} )) } \le \left( \frac{\epsilon_{g}}{\beta_{H}+2\epsilon_{H}} \right)^{1+\theta_1}
		,\]
	which is equivalent to
	\[
		\bar{m}_{x_k}(\eta) - \bar{m}_{x_k}(\eta^{*})
		\le \frac{\epsilon_{H}\epsilon_{g}^{2+2\theta_1} }{2(\beta_{H}+2\epsilon_{H})^{4+2\theta_1}}
		.\]
	By \cref{prop:krylov-tr} of \cref{prop:krylov} (\citep[Theorem 4.1]{carmon2020FirstOrderMethods}), the above criterion is met within at most
	\[
		C_{\mathrm{sub},2}\epsilon_{H}^{-1 /2} \log \left(\frac{\epsilon_{H}\epsilon_{g}^{2+2\theta_1} }{2(\beta_{H}+2\epsilon_{H})^{4+2\theta_1}}\right)^{-1}
	\]
	iterations.
	Letting $C_{\mathrm{sub}}^{\mathrm{tr}} \coloneqq C_{\mathrm{sub},2} \log\left(\frac{\epsilon_{H}\epsilon_{g}^{2+2\theta_1} }{2(\beta_{H}+2\epsilon_{H})^{4+2\theta_1}}\right)^{-1}$ gives the desired result.
\end{proof}

\ifSubfilesClassLoaded{\bibliography{ic}}{}
\section{Validating Iteration Complexity of RAR} \label{sec:valid-aroc}

In this subsection, we revisit \cref{sec:rar} while using \cref{alg:ar-sub} as the RAR subproblem process.
We highlight the only differences between \cref{alg:ar-sub} and a pure Krylov subspace method with \cref{eq:tc-1,eq:tc-2}:
a new termination criterion \cref{eq:tc-d} replacing \cref{eq:tc-2} (\cref{line:tcd})
and occasional recalculations of the iteration step using a regularized Hessian $\bar{H}_k = H_k + 2\epsilon_{H}I$ (\cref{line:ar-return-recal}).
The MEO does not impact the iteration complexity because we only use it to test the second-order stationarity for RAR.
Intuitively, the regularized Hessian will merely introduce a slightly larger Hessian operator norm bound $\beta_{H} + 2\epsilon_{H}$; and new termination criterion \cref{eq:tc-d}, which replaces \cref{eq:tc-2}, directly applies a condition on the model decrease, which will not only fulfill all previous discussions but also make their establishment more straightforward.

We begin by examining \cref{prop:g,prop:s,lem:g-bound}, which only rely on the Cauchy condition \cref{eq:tc-c}.
One can observe that in these propositions, as well as in \cref{lem:cdec}, any reliance on $H_k$ gets transformed into its corresponding norm bound $\beta_{H}$.
Therefore, the regularized Hessian merely leads to a slightly larger norm bound: $\beta_{H} + 2\epsilon_{H}$.
Without loss of generality, we can substitute all instances of $\beta_{H}$ with $\beta_{H}+2\epsilon_{H}$.
Furthermore, the corollary inequality \cref{eq:cauchy} of the Cauchy condition remains valid because
$$
	m_{x_k}(0) - m_{x_k}(\eta_k)
	= \bar{m}_{x_k}(0) - \bar{m}_{x_k}(\eta_k) + \epsilon_{H}\|\eta_k\|^2 + \varphi(\eta_k;\sigma_k)
	\ge \epsilon_{H}\|\eta_k\|^2 + \varphi(\eta_k;\sigma_k) \ge \varphi(\eta_k;\sigma_k).
$$

We then examine \cref{lem:sigmabar}. Recall that we have set $\omega = \mu = \nu$. Then, for \case{\rom{1}} in \cref{lem:sigmabar}, \cref{eq:u-2} directly gives a bound $\sigma_k \le C_{\sigma,2}$. Thus, we do not need to consider \case{\rom{1}.\rom{2}} where \cref{eq:tc-2} comes into play.
The bound of \case{\rom{2}} remains unchanged as it does not involve any termination criterion.
Therefore, using \cref{alg:ar-sub}, the upper bound of $\sigma_k$ reduces to
$$
	\bar{\sigma} = \kappa_3 (C_{\sigma,2} \vee C_{\sigma,5} \vee C_{\sigma,6}).
$$

For \cref{lem:obj-dec-ar}, a slight modification is required to accommodate the utilization of a regularized Hessian.
We opt for
a single lower bound on the model decrease for any $k\in \mathcal{S}_{GG}\sqcup \mathcal{S}_{L}$:
\begin{equation}\label{eq:obj-dec-ar-new}
	m_{x_k}(0) - m_{x_k}(\eta_k) \gtrsim \min \left\{  \epsilon_{g}^{\tfrac{2(2+\alpha)}{1+\alpha}}\epsilon_{H}^{- \tfrac{2+\alpha}{\alpha}}, \epsilon_{H}^{\tfrac{2+\alpha}{\alpha}} \right\}.
\end{equation}
It is evident that $\eta_k$ returned using \cref{eq:tc-d} directly satisfies this lower bound.
For $\eta_k$ returned using \cref{eq:tc-1} and an unregularized Hessian, since \cref{alg:ar-sub} also necessitates $\|g_k\| > \epsilon_{g}$ in this case, $k\in \mathcal{S}_{GG}\sqcup \mathcal{S}_{L}$ implies $k\in \mathcal{S}_{GG}$. \cref{lem:obj-dec-ar} for $k\in \mathcal{S}_{GG}$ directly gives \cref{eq:obj-dec-ar-new} without adaptation in this case, because
$$
	\left.\underbrace{\min \left\{ \epsilon_{g}^{\frac{2(2+\alpha)}{1+\alpha}}\epsilon_{H}^{-\frac{2+\alpha}{\alpha}}, \epsilon_{H}^{\frac{2+\alpha}{\alpha}} \right\}}_{\eqref{eq:obj-dec-ar-new}}
	\bigg/
	\underbrace{\vphantom{\min \left\{ \epsilon_{g}^{\frac{2(2+\alpha)}{1+\alpha}}\right\}}\epsilon_{g}^{\frac{2+\alpha}{1+\alpha}}}_{\textup{Lem.}~\ref{lem:obj-dec-ar}}\right.
	= \min \left\{ \epsilon_{g}^{\frac{2+\alpha}{1+\alpha}}\epsilon_{H}^{-\frac{2+\alpha}{\alpha}}, \left( \epsilon_{g}^{\frac{2+\alpha}{1+\alpha}}\epsilon_{H}^{-\frac{2+\alpha}{\alpha}} \right) ^{-1} \right\} \le 1.
$$
Finally, if $\eta_k$ is returned using a regularized Hessian, \cref{eq:ar-obj-dec-1} becomes
\begin{equation*}
	\epsilon_{g} \le (C_{s,1} + \bar{\sigma})\|\eta_k\|^{1+\alpha} + 2{\epsilon}_H\|\eta_k\|,
\end{equation*}
which resembles \cref{eq:s-dec-5}.
Thus, similar to \case{\rom{2}.\rom{1}} and \case{\rom{2}.\rom{2}}, we get $\|\eta_k\| \gtrsim \epsilon_{g}^{1 /(1+\alpha)}$ when $\|\eta_k\|^{\alpha} \gtrsim \epsilon_H$;
and when $\|\eta_k\|^{\alpha}\not\gtrsim \epsilon_{H}$, we have
$$
	\begin{aligned}
		\|\eta_k\| \ge & \left( \frac{-1 + \sqrt{1 + (C_{s,1}+\bar{\sigma})\epsilon_{g}{\epsilon}_H^{-\frac{1+\alpha}{\alpha}}}}{C_{s,1}+\bar{\sigma}} {\epsilon}_H^{\frac{1+\alpha}{2\alpha}} \right) ^{\frac{2}{1+\alpha}}                        \\
		\ge            & \left( \frac{-1 + \sqrt{1 + C_{s,1}+\bar{\sigma}}}{C_{s,1}+\bar{\sigma}} \min \left\{ \epsilon_{g}{\epsilon}_H^{-\frac{1+\alpha}{2\alpha}}, {\epsilon}_H^{\frac{1+\alpha}{2\alpha}} \right\} \right) ^{\frac{2}{1+\alpha}} \\
		\gtrsim        & \min \left\{ \epsilon_{g}^{\frac{2}{1+\alpha}}{\epsilon}_H^{-\frac{1}{\alpha}}, {\epsilon}_H^{\frac{1}{\alpha}} \right\}.
	\end{aligned}
$$
Then by the Cauchy condition \cref{eq:cauchy}, \cref{eq:obj-dec-ar-new} holds.
To recap, \cref{eq:obj-dec-ar-new} holds for all cases.
Therefore, the iteration complexity of RAR in \cref{cor:ar}
turns into
$$
	O\left( \max \left\{ \epsilon_{g}^{-\tfrac{2(2+\alpha)}{1+\alpha}}\epsilon_{H}^{\tfrac{2+\alpha}{\alpha}} , \epsilon_{H}^{-\tfrac{2+\alpha}{\alpha}} \right\}\right) .
$$
Combined with the operation complexity of Lanczos-based Krylov subspace methods and MEO, we get \cref{cor:ar-oc}.

\ifSubfilesClassLoaded{\bibliography{ic}}{}
\section{Validating Iteration Complexity of RTR} \label{sec:valid-troc}

In this subsection, we revisit \cref{sec:rtr} while using \cref{alg:tr-sub} as the RTR subproblem process.
We highlight the only difference between \cref{alg:tr-sub} and a pure Krylov subspace method with \cref{eq:tc-1}, \cref{eq:tc-2,eq:tc-t}: we will occasionally use $\eta^{\mathrm{E}}$ returned by an MEO with \cref{eq:tc-e} to construct the iteration step, on which we do not impose any other conditions such as the Cauchy condition \cref{eq:tc-c}.
Intuitively, we only invoke an MEO if there is uncertainty regarding the second-order stationarity, and it will return an eigenvector associated with $\lambda_{\min}(H_k)$ when $H_k \not\succeq -\epsilon_{H}I$. Similar to \cref{eq:tc-d}, this approach will not only fulfill our previous discussions but also make their establishment more straightforward.

We only need to examine \cref{lem:deltamin,lem:obj-dec-tr}.
According to \cref{alg:tr-sub}, an MEO is invoked only if \texttt{max\_flag = true} or $\|\xi_{j}\| < \Delta_k$. In both scenarios, we know that the first-order Krylov subspace solution resides within the interior of the trust region, and thereby $\|g_k\| \le (\beta_{H} + 2\epsilon_{H})\|\xi_1\|$ (see \cref{sec:pf-g-bound}).
Therefore, when using an MEO, \cref{eq:deltamin-1} in \cref{lem:deltamin} should be reformulated as
$$
	\frac{3}{4}(m_{x_k}(0) - m_{x_k}(\eta_k)) \le C_{d}\|\eta_{k}\|^{2+\alpha} + C_{R}(\beta_{H} + 2\epsilon_{H})\|\xi_1\|\|\eta_k\|^{1+\alpha},
$$
where we set $\alpha = \mu = \nu$. By \cref{eq:tc-e} and $\|\eta_k\| = \Delta_k$, we have
$$
	\frac{3}{8}\epsilon_{H}\Delta_{k}^2 \le (C_{d} + C_{R}(\beta_{H} + 2\epsilon_{H}))\Delta_k ^{2+\alpha},
$$
which gives $\Delta_k \gtrsim \epsilon_{H}^{1 /\alpha}$.
Therefore, \cref{lem:deltamin} still holds using an MEO, perhaps with different constants.

For \cref{lem:obj-dec-tr}, if $\eta_k$ is returned by an MEO, by \cref{eq:tc-e}, we directly have
$$
	m_{x_k}(0) - m_{x_k}(\eta_k) = -\Delta_{k}\left<g_k,\eta^{\mathrm{E}} \right> - \frac{\Delta_{k}^2}{2}\left<\eta^{\mathrm{E}}, H_k \eta^{\mathrm{E}} \right>
	\ge \frac{1}{4}\epsilon_{H}\Delta_{k}^2.
$$
Therefore, \cref{lem:obj-dec-tr} also holds for MEO.
In summary, the iteration complexity of RTR in \cref{cor:tr} remains valid when using \cref{alg:tr-sub}, as does the operation complexity in \cref{cor:tr-oc}.

\section{Computational Experiments Setup}
\label{apx:exp}

\subsection{Problem Setup}

This subsection establishes the connection between the nonconvex manifold optimization problem considered in \cref{sec:exp} and the quadratic maximization problem with a $L_2$ norm ball constraint.
The latter arises frequently in practical applications such as game intervention theory \citep{parise2023GraphonGames,kor2023Multiactivityinfluencea}, and is defined as follows:
\[
	\begin{aligned}
		\max_{x \in \mathbb{R}^{n+1}} \quad & \frac{1}{2}\left< x,Ax \right>_{E} - \left< x,b \right>_{E} \\
		\text{s.t.} \quad                   & \|x\|_2 \le C
		,\end{aligned}
\]
where $\left< \cdot ,\cdot  \right>_{E}$ is the Euclidean inner product, $A \in \mathbb{R}^{(n+1) \times (n+1)}$ is a positive semidefinite matrix, $b \in \mathbb{R}^{n+1}$ is a vector, and $C$ is a positive constant.
Since the objective is convex, the optimum is achieved at the boundary of the constraint set.
Hence, the problem is equivalent to the following unconstrained manifold optimization problem:
\[
	\begin{aligned}
		\min_{p \in \mathbb{S}^{n}} \quad & -\frac{C^{2}}{2}\left< p,Ap \right>_{E} + C\left< p,b \right>_{E}
		,\end{aligned}
\]
where the manifold $\mathbb{S}^{n}$ is the unit sphere in $\mathbb{R}^{n+1}$.
We further notice that
\[
	\left< p,b \right>_{E} = \|b\|_{2}\left< p,\frac{b}{\|b\|_{2}} \right>_{E}
	= \|b\|_{2}\cos \measuredangle (p,b)
	= \|b\|_{2}\dist(p,b /\|b\|_{2})
	.\]
Therefore, the problem is equivalent to a nonconvex manifold optimization problem with a distance regularization term.

Generalizing the regularization order, the problem becomes
\[
	\min_{p \in \mathbb{S}^{n}} \quad f(p) = f_{1}(p) + \dist(p,b)^{2+\mu}
	,\]
where $f_1(p) = -C/(2\|b\|_{2})\left< p,Ap \right>_{E}$ is smooth but not necessarily convex, and $\dist(p,b)^{2+\mu}\in C^{2,\mu}$.
This arrives at the problem considered in \cref{sec:exp}.

\subsection{Objective's Smoothness}

This subsection establishes that the objective function $f$ considered in \cref{sec:exp} belongs to the class $C^{2,\mu}$.
Note that in the normal neighborhood of $x_0$ (which is a half sphere in our case), the distance function can be expressed as
\[
	\dist (x_0,x) = \norm{\exp_{x_0}^{-1}(x)}_{E}
	,\]
where the norm is the Euclidean norm on the tagent space $T_{x_0}\M$.
As the exponential map $\exp$ is a smooth diffeomorphism, it suffices to inspect the smoothness of $\|\cdot \|_{E}^{2+\mu}$.
Let $f_2(x)\coloneqq\|x\|^{2+\mu}$, $x\in\R^{n}$. We have
\[
	\nabla^2 f_2(x) = (2+\mu)\|x\|^{\mu}  \left( I + \mu\|x\|^{-2}xx^{T}  \right)
	.\]
Thus,
\[
	\begin{aligned}
		\left\| \nabla^{2}f_2(x) - \nabla^{2}f_2(y) \right\|_{\mathrm{op}} = & (2+\mu)\left\| \|x\|^{\mu}\left( I + \mu\frac{x x^{T}}{\|x\|^{2}} \right) - \|y\|^{\mu}\left( I + \mu\frac{y y^{T}}{\|y\|^{2}} \right) \right\|_{\mathrm{op}} \\
		\le                                                              & \underbrace{(2+\mu)\left|\|x\|^{\mu} - \|y\|^{\mu}\right| \left\| I + \mu \frac{x x^{T}}{\|x\|^{2}} \right\|_{\mathrm{op}}}_{H_1}
		+ \underbrace{(2+\mu)\mu\|y\|^{\mu}\left\| \frac{x x^{T}}{\|x\|^{2}} - \frac{y y^{T}}{\|y\|^{2}} \right\|_{\mathrm{op}}}_{H_2}
		.\end{aligned}
\]
One can easily verify that the triangle inequality holds for
\[
	\left| \|x\|^{\mu} - \|y\|^{\mu} \right| \le \|x-y\|^{\mu}
	.\]
Further, define unit vectors $u = x /\|x\|$ and $v = y /\|y\|$. We have $\|uu^{T}\|_{\mathrm{op}} = 1$ and thus
\[
	H_1 \le (2+\mu)(1+\mu)\|x-y\|^{\mu}
	.\]
For $H_2$, note that $uu^{T}-vv^{T}$ is a rank-2 matrix with only two possibly nonzero eigenvalues $\lambda$ and $-\lambda$, as its trace is zero. 
Hence, $\lambda^{2}$ and $(-\lambda)^{2}$ are the only two eigenvalues of $(uu^{T}-vv^{T})^{2}$, and we have
\[
	\|uu^{T}-vv^{T}\|_{\mathrm{op}} 
	= \lambda = \sqrt{\frac{\lambda^{2} + (-\lambda)^{2}}{2}} = \sqrt{\frac{\tr((uu^{T}-vv^{T})^{2})}{2}} = \sqrt{1 - (u^{T}v)^{2}}
	= |\sin\theta|
,\]
where $\theta\in[0,\pi]$ is the angle between $u$ and $v$. 
Also note that $u$, $v$ determine a one-dimensional circle and $\|u-v\| = 2\sin(\theta/2) = \sin\theta / \cos(\theta/2) \ge \sin\theta$. Further,
\[
	\|u-v\| = \left\| \frac{x}{\|x\|}-\frac{y}{\|y\|} \right\|
	= \frac{\|x\cdot \|y\| - y\cdot \|x\|\|}{\|x\|\|y\|}
	\le \frac{|\|y\|-\|x\||}{\|y\|} + \frac{\|x-y\|}{\|y\|} \le \frac{2\|x-y\|}{\|y\|}
	.\]
Finally, since $|\sin\theta|\le 1$ and $\mu\le 1$, we get
\[
	|\sin\theta| \le |\sin\theta|^{\mu} \le \|u-v\|^{\mu} \le \frac{2^{\mu}}{\|y\|^{\mu}}\|x-y\|^{\mu}
	.\]
Pluggin this into $H_2$ gives
\[
	H_2 \le (2+\mu)\mu\|y\|^{\mu} \cdot \frac{2^{\mu}}{\|y\|^{\mu}}\|x-y\|^{\mu} = 2^{\mu}\mu(2+\mu) \|x-y\|^{\mu}
	.\]
Together, we have
\[
	\left\| \nabla^{2}f_2(x) - \nabla^{2}f_2(y) \right\|_{\mathrm{op}} \le (2+\mu)(1+\mu+2^{\mu}\mu) \|x-y\|^{\mu}
	.\]
We conclude that $f(p) = f_{1}(p) + f_{2}(\exp_{b}^{-1}(p))\in C^{2,\mu}$, where $f_1$ is smooth and $f_2$ is $C^{2,\mu}$.

\ifSubfilesClassLoaded{\bibliography{ic}}{}
\end{document}